\documentclass[numbers,webpdf]{ima-authoring-template}
\usepackage{algorithm}
\usepackage{algpseudocode} 
\usepackage{amsmath, amssymb}
\usepackage{amsthm}
\usepackage{graphicx}
\usepackage{subfig}
\usepackage{float}
\usepackage{tikz}
\usepackage{pgfplots}
\usepackage{caption}
\usepackage{url}
\usepackage{booktabs}
\usepackage{physics}
\usepackage{arxiv-preprint}

\newtheorem{theorem}{Theorem}[section]       % Theorem 1.1, 1.2, etc.
\newtheorem{lemma}{Lemma}[section]           % Lemma 1.1, 1.2, etc.
\newtheorem{proposition}{Proposition}[section]
\newtheorem{corollary}{Corollary}[section]

\theoremstyle{definition}
\newtheorem{definition}{Definition}[section]

\theoremstyle{remark}
\newtheorem{remark}{Remark}[section]

\everydisplay\expandafter{\the\everydisplay \small}
\begin{document}
\DOI{}
\title[Convergence of a Semi-Implicit Scheme]{Convergence analysis of a semi-implicit time discretization for the doubly degenerate Richards equation}
\author{Abderrahmane Benfanich*\address{\orgdiv{Department of Mathematics and Statistics}, \orgname{University of Ottawa}, \orgaddress{Ottawa, \country{Canada}}}}
\author{Yves Bourgault\address{\orgdiv{Department of Mathematics and Statistics}, \orgname{University of Ottawa}, \orgaddress{Ottawa, \country{Canada}}}}
\author{Abdelaziz Beljadid\address{\orgdiv{Department of Mathematics and Statistics}, \orgname{University of Ottawa}, \orgaddress{Ottawa, \country{Canada}} \newline \orgname{Mohammed VI Polytechnic University}, \orgaddress{\country{Morocco}}}}
\authormark{Benfanich et al.}
\corresp[*]{Corresponding author: \href{abenf099@uottawa.ca}{abenf099@uottawa.ca}}
\date{\today}
\journaltitle{Preprint}
\pubyear{2026}
\copyrightyear{2026}
\vol{0}
\access{}
\appnotes{Paper}
\copyrightstatement{}
\firstpage{1}

\abstract{A key contribution of this work is the derivation of convergence results for the semi-implicit time discretization in the general doubly degenerate setting. This paper presents a mathematical analysis of a doubly degenerate parabolic equation and its application to the Richards equation using a bounded auxiliary variable. We establish the existence of weak solutions using a semi-implicit time discretization combined with maximal monotone operator theory. The analysis is conducted within weighted Sobolev spaces, allowing for a rigorous treatment of the equation's strict degeneracy and strong nonlinearities. Unlike previous works that often rely on strict parabolicity or high regularity, we establish these results without imposing strictly positive lower bounds on the diffusivity. Furthermore, we prove that the Richards equation using the introduced auxiliary variable preserves the physical bounds of the saturation and demonstrate the unconditional linear convergence of the L-scheme linearization to the semi-discrete solution.}

\maketitle
% \tableofcontents
% \newpage

\section{Introduction}

Modeling fluid flow in porous media constitutes a fundamental challenge in fields ranging from hydrology and agricultural engineering to environmental science. The primary mathematical framework for describing water movement in the unsaturated zone is the \textit{Richards equation} \cite{Richards1931}. Derived by combining the principle of mass conservation with the Darcy-Buckingham law, this equation captures the complex nonlinear relationship between soil water content and pressure head. The resulting model is a degenerate parabolic partial differential equation, which presents significant analytical difficulties due to the strong nonlinearities in the hydraulic conductivity and the potential degeneracy of the diffusivity term in dry or fully saturated regimes.

The mathematical analysis of the Richards equation has been the subject of extensive research. Fundamental results regarding the existence of weak solutions for general quasilinear elliptic-parabolic equations were established by Alt and Luckhaus \cite{AltLuckhaus1983} using compactness methods and Kirchhoff transformations. The regularity of such solutions was further investigated by DiBenedetto \cite{DiBenedetto1993}, who established the Hölder continuity of the saturation. Regarding uniqueness, standard techniques often fail due to the lack of regularity; however, $L^1$-contraction principles developed by Otto \cite{Otto1996} and refined by Carrillo \cite{Carrillo1999} have provided a robust framework for proving uniqueness in the context of bounded domains.

Despite these theoretical advances, obtaining efficient and reliable numerical solutions remains a difficult task. As highlighted by List and Radu \cite{ListRadu2016}, the core challenge lies in the linearization of the discrete equations. The standard Newton method (or Newton-Raphson) is frequently employed due to its quadratic convergence rate. However, its convergence is only local; it often fails or suffers from severe time-step restrictions in regimes with low saturation or sharp wetting fronts where the derivative of the nonlinearity vanishes or explodes \cite{ListRadu2016, Jones2001}. The Picard iteration, while globally convergent under certain conditions, is often prohibitively slow. 

To address the numerical difficulties arising from the degeneracy, various alternative strategies have been proposed. One prominent class of methods is the Primary Variable Switching (PVS) approach. In this framework, the numerical scheme dynamically selects either the pressure head or the saturation as the primary unknown, depending on the local flow regime \cite{Diersch1999, Forsyth1995, Wu2001}. Typically implemented within Newton-Raphson solvers, PVS constructs the Jacobian matrix based on derivatives with respect to the currently active variable \cite{Brunner2012, Krabbenhoft2007}. While applicable to a wide range of problems, PVS approaches often suffer from non-smooth transitions between variables. As noted in \cite{Krabbenhoft2007} and \cite{Zha2017}, these discontinuities can yield physically unrealistic solutions, particularly at the sharp interface between saturated and unsaturated zones. Although smoothing techniques and refined switching criteria have been developed to mitigate these oscillations \cite{Maina2017, Kees2002}, their effectiveness remains highly problem-dependent (see \cite{Zha2019} for a review).

A distinct strategy for handling the nonsmooth nature of the Richards equation is regularization. This approach replaces the degenerate constitutive relationships with smoothed, non-degenerate approximations controlled by a regularization parameter. Literature in this area is extensive, ranging from parabolic regularizations for dry-region unsaturated flow \cite{Schweizer2007} to schemes specifically tailored for doubly degenerate equations \cite{Pop2011}. More recently, adaptive regularization frameworks guided by a posteriori error estimators have been proposed \cite{Fevotte2024}. While regularization significantly improves the robustness of numerical solvers, it inherently modifies the underlying governing equation, introducing an artificial modeling parameter that must be carefully tuned to balance physical accuracy with numerical stability.

Returning to linearization schemes, Pop et al. \cite{Pop2004} and Slodička \cite{Slodicka2002} independently proposed the L-scheme to bridge the gap between robustness and efficiency. This method is a stabilized fixed-point iteration; by adding a stabilization term, the L-scheme guarantees unconditional stability and global convergence, making it particularly attractive for degenerate problems. However, a critical examination of the literature reveals that the theoretical convergence analysis for these schemes often relies on assumptions that contradict the physical reality of the problem. Many proofs assume that the diffusivity is strictly bounded from below or that the exact solution possesses high regularity, which is not the case for the solution of this type of doubly degenerate parabolic equation.

Furthermore, recent efforts have focused on addressing the convergence speed of the L-scheme. The standard L-scheme requires the stabilization parameter $L$ to be greater than or equal to the Lipschitz constant of the nonlinearity ($L \ge \sup |b'|$) \cite{Pop2004,Slodicka2002}. In the degenerate Richards equation, where the derivative of the relationship between saturation and capillary pressure can become very large, this constraint leads to excessive stabilization and slow convergence. To mitigate this, Mitra and Pop \cite{Mitra2019} introduced the Modified L-scheme. Further enhancements include dynamic strategies, where the regularization parameter is adapted during the iteration process. For instance, dynamic regularization strategies \cite{Fevotte2024} adjust the smoothing parameter based on error reduction, while Anderson acceleration techniques have been applied to the L-scheme to improve its convergence rate in stiff regimes.

Semi-implicit IMEX methodologies offer a powerful alternative to expensive fully implicit schemes. By linearizing terms like hydraulic conductivity using extrapolation or Taylor expansion, they achieve accuracy and efficiency comparable to Newton's method while demonstrating superior robustness in handling degenerate relationships \cite{Kamil2024, Paniconi1991, Keita2021}. However, despite these strengths, these schemes present specific drawbacks, including conditional stability that may necessitate small time steps when gravity dominates, and a sensitivity to free parameters that often requires mass lumping or regularization to ensure physical realism \cite{Celia1990, Kamil2024, Keita2021}. Crucially, a significant theoretical gap persists alongside these numerical challenges: the literature lacks rigorous convergence proofs that account for the strict degeneracy of the Richards equation without relying on unphysical regularity assumptions.

Recently, a numerical framework was introduced in \cite{benfanich2025} to address these computational challenges. By recasting the Richards equation using a continuous bounded auxiliary variable, this methodology maps the potentially unbounded physical variables into a strictly bounded interval. This formulation naturally enforces the physical constraints of the saturation without requiring artificial limiters or primary variable switching. Furthermore, this transformation was coupled with a semi-implicit backward Euler time discretization and a fully mass-lumped finite element method to provide a stable numerical treatment for challenging regimes, such as dry soils, sharp wetting fronts, and high infiltration rates. While the numerical efficacy of this approach was shown in \cite{benfanich2025}, the rigorous mathematical foundation justifying its convergence, well-posedness, and physical fidelity in the fully degenerate regime remained is not studied. The present work bridges this theoretical gap. Our overarching objective is to provide a comprehensive mathematical analysis of this doubly degenerate formulation that explicitly overcomes the analytical limitations of existing theoretical frameworks.

The theoretical framework underpinning our analysis to close this gap is the method of discretization in time, classically known as Rothe's method \cite{Rektorys1982, Kacur1990}. Unlike the standard Method of Lines which discretizes space first to obtain a system of ordinary differential equations, Rothe's method discretizes the time variable first, approximating the evolution equation by a sequence of stationary elliptic boundary value problems at each time step. This technique serves as both a numerical scheme and a constructive proof method for establishing the existence and uniqueness of solutions to nonlinear parabolic problems. Its primary advantage in the context of degenerate equations lies in its reliance on compactness arguments rather than strong regularity estimates. By constructing piecewise constant and piecewise linear interpolants in time (Rothe functions) and establishing uniform a priori estimates, one can extract convergent subsequences that satisfy the continuous problem in the weak sense, even when the solution lacks the smoothness required for standard error analysis.

\subsection{Main Contributions}
In this study, we address the theoretical challenges of the Richards equation by introducing a specific formulation that handles the double degeneracy within the framework of time discretization. While the well-posedness of degenerate parabolic equations has been studied classically (e.g., Alt and Luckhaus \cite{AltLuckhaus1983}), our theoretical contributions differ from the existing literature in four key areas:
\begin{itemize}
    \item \textbf{Convergence of semi-implicit discretization in strictly degenerate regimes:} While existing semi-implicit methods are used to solve the Richards equation (e.g., \cite{Kamil2024, Keita2021, benfanich2025}), to our knowledge there is no theoretical study of convergence or well-posedness for these semi-implicit methods in the doubly degenerate case. We provide a rigorous convergence proof for a semi-implicit time discretization applied to the fully doubly degenerate setting, without assuming a non-vanishing lower bound on hydraulic conductivity.
    \item \textbf{Natural treatment of strict degeneracy via weighted Sobolev spaces:} Standard analyses often rely on classical $H^1$ Sobolev spaces, which become ill-defined in strictly dry or saturated regimes where $K(u)=0$. Here, we establish existence, uniqueness, and convergence analysis in weighted Sobolev spaces \cite{Cavalheiro2008WeightedSobolev}, yielding a framework that handles strict spatial degeneracy without artificially modifying the governing equations.
    \item \textbf{Unconditional linear convergence of the L-scheme for fully degenerate problems:} Existing convergence analyses for the L-scheme (e.g., \cite{Pop2004}) generally rely on non-degenerate diffusivity or high-regularity assumptions. We prove unconditional linear convergence of the L-scheme for the fully degenerate semi-discrete problems, without strict parabolicity or artificial regularization.
    \item \textbf{Rigorous preservation of physical bounds (maximum principle):} Using the bounded auxiliary variable, we establish a generalized physical maximum principle for the continuous equation and prove preservation of strict physical saturation bounds ($0 \le S \le 1$), including settings with state-dependent source terms and heterogeneous media with jump discontinuities.
\end{itemize}

Our analysis proceeds in several steps using Rothe's method:
\begin{enumerate}
    \item We discretize the equation in time using a semi-implicit backward Euler scheme.
    \item We prove that the resulting degenerate elliptic problems admit a unique solution in a weighted Sobolev space, utilizing maximal monotone operator theory \cite{brezis1973ope}.
    \item We prove that the {L-scheme} converges linearly and unconditionally to the solution of these semi-discrete problems, independent of the degeneracy type.
    \item We demonstrate the convergence of the semi-implicit method to the weak solution of the continuous equation, thereby validating the numerical scheme.
    \item Finally, we prove a generalized maximum principle for the continuous equation, ensuring that the exact solution respects physical bounds even in the presence of bounded source terms and spatial heterogeneity.
\end{enumerate}
While the primary objective of this work is the rigorous convergence analysis of the semi-implicit L-scheme (Sections 8 and 9), we first establish the continuous well-posedness and a generalized maximum principle (Sections 5 and 6) to ensure the physical fidelity of the exact solution being approximated. We then provide a focused numerical validation of these theoretical results in Section 10.
\subsection{Presentation of the Model}
The present study is a follow-up of our recent work \cite{benfanich2025} where we introduced a new bounded auxiliary variable to solve the Richards equation. We will focus on the mathematical analysis of the obtained doubly degenerate equation.

Let $\Omega \subset \mathbb{R}^d$, with $d \in \{1, 2, 3\}$, be a bounded open set representing the domain of a porous medium, such that the boundary $\partial \Omega$ is Lipschitz. Let $T > 0$, and $I = (0, T)$ denote the time interval. The classical formulation of the Richards equations is given by
\begin{equation}
\label{eq:richards}
\frac{\partial \theta}{\partial t} + \nabla \cdot \boldsymbol{q} = 0, \quad \text{in } \Omega \times I,
\end{equation}
where $\theta = \theta(\boldsymbol{x}, t)$ is the volumetric water content and the water flux $\boldsymbol{q}$ is described by the Darcy-Buckingham law \citep{Richards1931}:
\begin{equation}
\label{eq:darcy}
\boldsymbol{q} = -K_s(\boldsymbol{x}) K_r(\boldsymbol{x}, S) \nabla (\Psi(\boldsymbol{x}, S) + z).
\end{equation}
Here, $K_s$ is the saturated hydraulic conductivity, $K_r$ is the relative permeability, and $\Psi$ represents the pressure head (or capillary suction). The relationship between the pressure head $\Psi$ and the saturation $S$ is determined by empirical constitutive models. Additionally, $\boldsymbol{x} = (x, z)^T$ denotes the spatial coordinates with the vertical coordinate $z$ oriented positively upward. Let $\theta_r$ and $\theta_s$ denote the residual and saturated water contents, respectively, such that $\theta_r \leq \theta \leq \theta_s$. We introduce the effective saturation $S$, defined as:
\[
S = \frac{\theta - \theta_r}{\theta_s - \theta_r}.
\]
Consequently, we have $0 \leq S \leq 1$. Substituting the physical variables with the normalized saturation leads to the saturation based equation. To simplify the notation in the subsequent mathematical analysis, we will henceforth denote the effective saturation $S$ by the variable $\theta$, where:
\begin{equation*}
0 \leq \theta(\boldsymbol{x}, t) \leq 1, \quad \forall (\boldsymbol{x}, t) \in \Omega \times I.
\end{equation*}
\subsection{Notation and Definitions of Spaces}

In this section, we collect the notations, definitions of functional spaces, and norms used throughout the paper.

\subsection*{General Notation}
\begin{itemize}
    \item $\Omega \subset \mathbb{R}^d$: A bounded open set with a Lipschitz boundary, where $d \in \{1, 2, 3\}$.
    \item $T > 0$: The final time.
    \item $I = (0, T)$: The time interval.
    \item $\textbf{x}  \in \Omega$: Spatial coordinate.
    \item $\langle f, v \rangle_{\mathcal{H}' \times \mathcal{H}}$: The duality pairing between a functional $f \in \mathcal{H}'$ and an element $v \in \mathcal{H}$.
    \item Let $g:\mathbb{R} \to \mathbb{R}$. We define the induced operator on function spaces by $g(v)(x) = g(v(x))$, for $v$ in some real function space.
\end{itemize}

\subsection*{Functional Spaces and Norms}
In this section, we introduce the functional setting and notations used throughout the analysis. We adopt the following definitions:

    \begin{itemize}
        \item $C^k(\bar \Omega)$, where $k \in \mathbb{N} \cup \{0,\infty\}$: The space of functions possessing continuous partial derivatives up to order $k$ in $\Omega$. $C^0(\bar \Omega)$ denotes the space of continuous functions.
        \item $C_c^\infty(\Omega)$: The space of infinitely differentiable functions with compact support contained in $\Omega$.
        \item For a real Banach space $\mathcal C$, we denote by $\mathcal C'$ the dual space of $\mathcal C$, defined by the set of continuous linear forms.
        \item $L^2(\Omega)$: The space of square-integrable functions with inner product $$(u,v) = \int_\Omega uv \, d\textbf{x}$$ and norm $$\|u\| = {(u,u)^{\frac 12}}.$$
        \item $\mathcal{H} = H_0^1(\Omega)$: The Sobolev space of functions with square-integrable derivatives vanishing on the boundary. By the Poincaré inequality, we equip this space with the inner product $$(u,v)_1 =  \int_\Omega \nabla u \cdot \nabla v \;d\textbf{x}$$ and the norm $$\|u\|_1 = \left( \int_\Omega |\nabla u|^2 \;d\textbf{x} \right)^{1/2}.$$
        \item $\mathcal{H}' = H^{-1}(\Omega)$: The dual space of $\mathcal{H}$ with the norm defined as $$\|f\|_{-1} = \sup_{v \in \mathcal{H}, \|v\|_1 \neq 0} \frac{|\langle f, v \rangle_{\mathcal{H}' \times \mathcal{H}}|}{\|v\|_1}.$$
        \item $X = L^2(I; \mathcal{H})$: The Bochner space of square-integrable functions from $I$ to $\mathcal{H}$ with the norm
        $$
        \|u\|_X = \left( \int_I \|u(t)\|_{\mathcal{H}}^2 \;dt \right)^{1/2}.
        $$
        \item $X' = L^2(I; \mathcal{H}')$: The dual space of $X$ with the norm
        $$
        \|u\|_{X'} = \left( \int_I \|u(t)\|_{\mathcal{H}'}^2 \;dt \right)^{1/2}.
        $$
        \item $C_w(I; L^2(\Omega))$: The space of weakly continuous functions $u \in L^\infty(I; L^2(\Omega))$ such that for every $v \in L^2(\Omega)$, the mapping $t \mapsto (u(t), v)$ is continuous on $\bar I$.
    \end{itemize}

\subsection*{Operators and Convergence}
\begin{itemize}
    \item \textbf{Weak Convergence:} A sequence $\{u_n\}_{n\in \mathbb N}$ converges weakly to $u$ in a Banach space $\mathcal{C}$, denoted $u_n \rightharpoonup u$, iff 
    $$
    \ell(u_n) \to \ell(u), \quad \forall \ell \in \mathcal{C}^*.
    $$
    In particular, if $\mathcal{C}$ is a Hilbert space, weak convergence is equivalent to 
    $$
    (u_n,v)_{\mathcal{C}} \to (u,v)_{\mathcal{C}}, \quad \forall v \in \mathcal{C}.
    $$
    \item \textbf{Weak-$*$ Convergence:} A sequence $\{u_n\}_{n\in \mathbb N}$ converges weak-$*$ to $u$ in a dual space $\mathcal{C}'$, denoted $u_n \overset{*}{\rightharpoonup} u$, if and only if
    $$
    \langle u_n, x \rangle \to \langle u, x \rangle, \quad \forall x \in \mathcal{C}.
    $$
    In particular, for $\mathcal{C}' = L^\infty(I; L^2(\Omega))$ (the dual of the space $\mathcal{C} = L^1(I; L^2(\Omega))$), this is equivalent to
    $$
    \int_I (u_n(t), \varphi(t))_{L^2(\Omega)} \, dt \to \int_I (u(t), \varphi(t))_{L^2(\Omega)} \, dt, \quad \forall \varphi \in L^1(I; L^2(\Omega)).
    $$
    Note that if $\mathcal{C}$ is a reflexive Banach space (e.g., a Hilbert space), weak and weak-$*$ convergence are equivalent. By the Banach-Alaoglu Theorem, the closed unit ball in the dual space is weak-$*$ compact.
\end{itemize}
\subsection*{Important equalities and inequalities}
\begin{itemize}
    \item \textbf{Algebraic Identity:} For all $x,y \in \mathbb{R}$, the following identity holds:
    $$
    (x-y)x = \frac{1}{2} (x^2 - y^2 + (x-y)^2).
    $$
    \item \textbf{Young's inequality:} For all $x,y \in \mathbb{R}$ and for all $\delta > 0$, we have:
    $$
    |xy| \leq \frac{1}{p\delta^{\frac p2}}|x|^p + \frac{\delta^{\frac q2}}{q}|y|^q,
    $$
    for $p,q \geq 1$ and $\frac 1p+\frac 1q = 1$.
    
    \item \textbf{Cauchy-Schwarz inequality:} Let $f,g \in L^2(\Omega)$. Then:
    $$
    |(f,g)| \leq \|f\|\|g\|.
    $$
    
    \item \textbf{H\"older's inequality:} Let $f\in L^p(\Omega)$ and $g \in L^q(\Omega)$ with $p,q \geq 1$ such that $\frac{1}{p}+\frac{1}{q}=1$. Then:
    $$
    |(f,g)| \leq \|f\|_p\|g\|_q,
    $$
    where $\|f\|_p$ and $\|g\|_q$ denote the $L^p$ and $L^q$ norms, respectively.
    \item \textbf{The Poincaré inequality:} There exists $C_P>0$ such that for all $v \in H^1_0(\Omega)$ we have 
    $$
    \|v\| \leq C_P \|\nabla v\|.
    $$
\end{itemize}
\section{Existence of Weak Solutions}
In this work, we consider a doubly degenerate parabolic equation that generalizes the $u$-formulation for the Richards equation originally introduced in \cite{benfanich2025} for homogeneous domains. Using the normalized variable $\theta$, the governing equation, supplemented with initial and homogeneous Dirichlet boundary conditions, is expressed as:
\begin{equation}
    \label{uform}
    \left\{
    \begin{aligned}
        &\frac{\partial \theta(u)}{\partial t} - \nabla \cdot (K(u) \nabla u) - \nabla\cdot \bar K(\textbf{x},t,u)=\mathcal{S}(\textbf{x},t,u) && \text{in } \Omega \times I, \\
        &u(\cdot, 0) = u_0 && \text{in } \Omega, \\
        &u = 0 && \text{on } \partial\Omega \times I.
    \end{aligned}
    \right.
\end{equation}
where $K$ and $\bar K$ are expressed in terms of hydraulic parameters of the medium, and $\mathcal{S}$ represents a non-linear source term.
\begin{definition}
A function $u$ is called a \textit{weak solution} of \eqref{uform} if and only if $u \in L^\infty(I; L^2(\Omega))$, $\theta(u(0))=\theta(u_0)$ in $L^2(\Omega)$, $\Phi(u) \in X = L^2(I; \mathcal{H})$, and $\frac{\partial \theta(u)}{\partial t} \in X' =L^2(I; \mathcal{H}')$, where $\mathcal{H}=H^1_0(\Omega)$, such that
\begin{equation}
    \int_I \left\langle\frac{\partial \theta(u)}{\partial t}, v\right\rangle_{\mathcal{H}' \times \mathcal{H}}dt + \int_I (\nabla \Phi (u), \nabla v)dt + \int_I \bigg(\bar K(u), \nabla v\bigg) dt = \qty(\mathcal S(u),v), \quad \forall v \in X,
    \label{mainequation}
\end{equation}
where $\Phi$ is defined as
\begin{equation}
    \Phi(s) = \int_0^s K(\xi) d\xi.
    \label{defphi}
\end{equation}
\end{definition}

Assume that we can extend the functions $K$, $\bar K$ and $\theta$ to $\mathbb R$, such that they satisfy the following assumptions

{\renewcommand{\labelenumi}{(H\arabic{enumi})}
\begin{enumerate}
    \item \label{H1} The function $\theta : \mathbb{R} \to \mathbb{R}$ satisfies the following conditions:
    \begin{enumerate}
        \item \textbf{Monotonicity and Regularity:} $\theta$ is strictly increasing with $\theta(0)=0$. We assume $\theta \in C^1(\mathbb{R})$ and its derivative $\theta'$ is bounded.
        
        \item \textbf{Inverse Regularity:} The inverse function $\theta^{-1}:\mathbb R\to \mathbb R$ exists and belongs to $W^{1,1}_{Loc} (\mathbb R)$.
        
        \item \textbf{Continuity Condition:} There exist constants $0<\delta\leq 1$ and $H_\theta >0$ such that for all $\zeta,\eta \in \mathbb{R}$:
        $$|\zeta-\eta| \leq H_\theta\left(|\theta(\zeta) -\theta(\eta)|+|\theta(\zeta) -\theta(\eta)|^\delta\right).$$
    \end{enumerate}
    \item \label{H2} The function $\theta$ satisfies the growth condition $\theta(\zeta)\zeta \geq \alpha \zeta^2$ for all $\zeta \in \mathbb R$, with a constant $\alpha >0$.
    \item \label{H3}
The functions $K : \mathbb{R} \to \mathbb{R}$ and $\bar{K} : \Omega \times I \times \mathbb{R} \to \mathbb{R}^d$ are assumed to be non-negative, bounded by a constant $M > 0$ component-wise, and satisfy the Carathéodory conditions (measurable in $\textbf{x}$, continuous in $(t, u)$). 

Furthermore, we impose the following structural conditions:
\begin{enumerate}
    \item \textbf{Decomposition:} The vector field $\bar{K}$ admits the factorization:
    \begin{equation}
        \bar{K}(\textbf{x},t,u) = K(u) \bar{K}_1(\textbf{x},t,u),
    \end{equation}
    where the auxiliary function $\bar{K}_1$ is non-negative component-wise, continuous in $(t, u)$, measurable in $\textbf{x}$, and uniformly bounded by $M$.
    
    \item \textbf{Degeneracy Regime:} The set where $K$ vanishes is negligible. We assume that:
    $$
    \lim_{\delta \to 0} \|\mathbf{1}_{\{K < \delta\}}\|_{L^1(\mathbb R)} = 0.
    $$
\end{enumerate}
 \item \label{H4} Assume that $\mathcal{S}: \Omega \times I \times \mathbb{R} \to \mathbb{R}$ is a Carathéodory function, specifically:
    \begin{enumerate}
        \item \textbf{Measurability:} For every $\eta \in \mathbb{R}$, the function $(\textbf{x},t) \mapsto \mathcal{S}(\textbf{x},t,\eta)$ is measurable.
        \item \textbf{Continuity and Boundedness:} For almost every $\textbf{x} \in \Omega$, the function $(t,\eta) \mapsto \mathcal{S}(\textbf{x},t,\eta)$ is continuous and bounded with a constant $M_S>0$.
    \end{enumerate}
    
\end{enumerate}
}
\begin{remark}
    Under Hypothesis \ref{H3}, the function \(\Phi\) is an increasing $C^1(\mathbb R)$ function with bounded derivative.
\end{remark}

For the remainder of this paper, we don't denote the dependence of the source term on spatial and temporal variables for notational brevity.
\begin{theorem}
There exists a weak solution to the equation \eqref{uform}.
\label{theorem1}
\end{theorem}

The proof of Theorem \ref{theorem1} is established through several steps in the subsequent sections.
\section{Weighted Sobolev Spaces}
Suppose we have a weight $\omega: \Omega \to \mathbb{R}$ such that $\omega$ is measurable, non-negative, and bounded. For $\varphi \in C^{\infty}_c(\Omega)$, we define the semi-norm
\begin{equation*}
    |\varphi|_V = \sqrt{\int_\Omega \omega^2(\textbf{x}) |\nabla \varphi|^2\, d\textbf{x}}.
\end{equation*}
We define the associated norm by
\begin{equation*}
    \|\varphi\|_V = \sqrt{\|\varphi\|^2 + |\varphi|_V^2},
\end{equation*}
which is indeed a norm on $C^{\infty}_c(\Omega)$. We set $V = \overline{C^\infty _c (\Omega)}^{\|\cdot\|_V}$ as the completion of $C^\infty _c (\Omega)$ with respect to the norm $\|\cdot\|_V$. By the Poincaré inequality, we have the following continuous embeddings:
\begin{equation}
    \mathcal{H} \subset V \subset L^2(\Omega).
\end{equation}
Furthermore, if the weight is uniformly bounded away from zero, we have:
\begin{equation}
    V = \mathcal{H}.
\end{equation}
For more information about these spaces, see \cite{Cavalheiro2008WeightedSobolev}.

The space $V$ is a Hilbert space equipped with the following inner product:
\begin{equation}
    (u,v)_V = (u,v) + (u, v)_{\omega},
\end{equation}
where
\begin{equation}
    (u,v)_\omega = (\omega(\textbf{x}) \nabla u, \omega(\textbf{x})\nabla v).
\end{equation}
\section{Semi-discretized Richards' Equation in Time}

In this section, we consider a semi-discretization of equation \eqref{uform} in time and subsequently demonstrate that the resulting problem admits a unique solution in the space $V$.

We employ a semi-implicit Euler scheme for the time discretization. The time interval $I = (0, T)$ is divided into $N \in \mathbb{N}$ sub-intervals of equal length $\tau = \frac{T}{N}$. We define the discrete time points $t_n = n\tau$ for $0 \leq n \leq N$. Let $u_n$ denote the approximation of $u(t_n)$. For $n \geq 1$, given $u_{n-1}$, we seek $u_n \in V$ satisfying
\begin{equation}
    (\theta(u_n), v) - (\theta(u_{n-1}), v) + \tau (u_n, v)_\omega + \tau \left(\omega \bar K_1^{n-1}(u_{n-1}), \omega \nabla v \right) = \tau(\mathcal{S}^{n-1}(u^{n-1}),v), \quad \forall v \in V,
    \label{semidiscfull}
\end{equation}
where the weight function is defined as
\begin{equation}
    \omega(\textbf{x}) = \sqrt{K(u_{n-1})}.
\end{equation}
in the rest of the paper we drop the index for $\mathcal S^{n-1}(u^{n-1})=S(.,t_{n-1}, u^{n-1}), \bar K^{n-1}(u^{n-1})= K(.,t_{n-1}, u^{n-1}), \bar K_1^{n-1}(u^{n-1}) = \bar K_1 (.,t_{n-1}, u^{n-1})$ and denote them by $\mathcal S, \bar K, \bar K_1$, respectively.

Equation \eqref{semidiscfull} can be rewritten in the following variational form: find $u_n \in V$ such that
\begin{equation}
    (\theta(u_n), v) + \tau (u_n, v)_\omega = \langle f, v \rangle_{V' \times V}, \quad \forall v \in V,
    \label{semidesceq}
\end{equation}
where $f \in V'$ is well-defined due to hypothesis \ref{H3} and \ref{H4}.

To prove existence, we first regularize the equation by introducing a parameter $\epsilon > 0$ and defining the regularized water saturation function $\theta_{\epsilon}(u) = \theta(u) + \epsilon u$. The regularized problem is to find $u_{n,\epsilon} \in V$ such that
\begin{equation}
    (\theta(u_{n,\epsilon}), v) + \epsilon(u_{n,\epsilon}, v) + \tau (u_{n,\epsilon}, v)_\omega = \langle f, v \rangle_{V' \times V}, \quad \forall v \in V.
    \label{regsemidesc}
\end{equation}

\begin{proposition}
    For all $\epsilon > 0$, the regularized equation \eqref{regsemidesc} possesses a unique solution in $V$. Furthermore, there exists a constant $C > 0$, independent of $\epsilon$, such that
    $$ \|u_{n,\epsilon}\|_V \leq C, \quad \forall \epsilon > 0. $$
    \label{prop:existofreg}
\end{proposition}

\begin{proof}
    Let $\epsilon>0$. We employ the $L$-scheme, introduced in \cite{Pop2004}, to establish existence. Let $L > 0$ and an initial guess $u \in V$ be given. We define $Tu \in V$ as the solution to the linear problem:
    \begin{equation}
        L(Tu - u, v) + (\theta(u), v) + \epsilon(u, v) + \tau (Tu, v)_\omega = \langle f, v \rangle_{V' \times V}, \quad \forall v \in V.
        \label{Lschemop}
    \end{equation}
    The operator $T: V \to V$ is well-defined. Indeed, problem \eqref{Lschemop} can be expressed as
    \begin{equation}
        a_L(Tu, v) = \chi_u(v), \quad \forall v \in V,
    \end{equation}
    where the bilinear form $a_L$ and the linear functional $\chi_u$ are defined by
    \begin{align*}
        a_L(w, v) &= L(w, v) + \tau (w, v)_\omega, \\
        \chi_u(v) &= \langle f, v \rangle + (L - \epsilon)(u, v) - (\theta(u),v).
    \end{align*}
    The linear form $\chi_u$ is bounded on $V$, and $a_L$ is a bounded bilinear form. Moreover, $a_L$ is coercive:
    \begin{equation}
        a_L(w, w) = L \|w\|^2 + \tau |w|_V^2 \geq \min(L, \tau) \|w\|_V^2.
    \end{equation}
    Thus, by the Lax-Milgram theorem \cite{LaxMilgram+1955+167+190}, for every $u \in V$, there exists a unique $Tu \in V$ satisfying \eqref{Lschemop}.

    We now show that $T$ is a contraction mapping on $V$. If so, $T$ has a unique fixed point $u \in V$, which is the unique solution to \eqref{regsemidesc}. Let $u_1, u_2 \in V$. Subtracting the equations for $Tu_1$ and $Tu_2$, we obtain
    \begin{equation}
        L(Tu_1 - Tu_2, v) - L(u_1 - u_2, v) + (\theta_\epsilon(u_1) - \theta_\epsilon(u_2), v) + \tau(Tu_1 - Tu_2, v)_\omega = 0, \quad \forall v \in V.
    \end{equation}
    Testing with $v = Tu_1 - Tu_2$ yields
    \begin{equation}
        L\|Tu_1 - Tu_2\|^2 + \tau |Tu_1 - Tu_2|_V^2 = L(u_1 - u_2, Tu_1 - Tu_2) - (\theta_\epsilon(u_1) - \theta_\epsilon(u_2), Tu_1 - Tu_2).
    \end{equation}
    Since $\theta \in C^1$, by the Mean Value Theorem, there exists $\xi$ with values between $u_1$ and $u_2$ such that
    \begin{equation}
        \theta_\epsilon(u_1) - \theta_\epsilon(u_2) = \theta'_\epsilon(\xi)(u_1 - u_2) \text{ a.e}.
    \end{equation}
    We have 
    \begin{equation}
        \epsilon \leq \theta'_\epsilon \leq L_\theta + \epsilon,
        \label{boundonthetaeps}
    \end{equation}
    where $L_\theta = \sup \theta' > 0$. Choosing $L > L_\theta + \epsilon$, we estimate the right-hand side by Cauchy-Schwarz and \eqref{boundonthetaeps}:
    \begin{align*}
        L\|Tu_1 - Tu_2\|^2 + \tau |Tu_1 - Tu_2|_V^2 &\leq \int_\Omega (L - \theta'_\epsilon(\xi)) |u_1 - u_2| |Tu_1 - Tu_2| \, d\textbf{x} \\
        &\leq (L - \epsilon) \|u_1 - u_2\| \|Tu_1 - Tu_2\|.
    \end{align*}
    Dividing by $L$ and introducing the equivalent norm $\|u\|_L^2 = \|u\|^2 + \frac{\tau}{L}|u|_V^2$, we obtain
    \begin{equation}
        \|Tu_1 - Tu_2\|_L \leq \frac{L - \epsilon}{L} \|u_1 - u_2\|_L.
    \end{equation}
    Since $\frac{L - \epsilon}{L} < 1$, $T$ is a contraction. By the Banach Fixed Point Theorem \cite{Agarwal2018}, $T$ has a unique fixed point $u_{n,\epsilon}$.

    To establish a uniform bound, we test \eqref{regsemidesc} with $v = u_{n,\epsilon}$:
    \begin{equation*}
        (\theta(u_{n,\epsilon}), u_{n,\epsilon}) + \epsilon \|u_{n,\epsilon}\|^2 + \tau |u_{n,\epsilon}|_V^2 = \langle f, u_{n,\epsilon} \rangle.
    \end{equation*}
    Using hypothesis \ref{H2}, we have
    $$ \min(\alpha, \tau) \|u_{n,\epsilon}\|_V^2 \leq \alpha \|u_{n,\epsilon}\|^2 + \tau |u_{n,\epsilon}|_V^2 \leq \langle f, u_{n,\epsilon} \rangle \leq \|f\| \|u_{n,\epsilon}\|_V. $$
    Therefore, for all $\epsilon > 0$,
    \begin{equation}
        \|u_{n,\epsilon}\|_V \leq \frac{\|f\|}{\min(\alpha, \tau)}.
    \end{equation}
\end{proof}

\begin{theorem}
    The semi-discretized Richards equation \eqref{semidesceq} admits a unique solution in $V$.
    \label{eistenceforsemidis}
\end{theorem}

To establish the proof of Theorem \ref{eistenceforsemidis}, we recall the following definition and lemma from monotone operator theory.

\begin{definition}[Proposition 2.2 in \cite{brezis1973ope}]
    Let $H$ be a Hilbert space and $A: H \to H$ an operator. We say that $A$ is \textit{monotone} if $(Ax - Ay, x - y) \geq 0$ for all $x, y \in H$. We say that $A$ is \textit{maximal monotone} if $A$ is monotone and $\operatorname{range}(I + A) = H$.
\end{definition}

The following lemma provides a convergence result for maximal monotone operators, which is essential for identifying the limit of the nonlinear terms.

\begin{lemma}[Proposition 2.5 in \cite{brezis1973ope}]
    Let $H$ be a Hilbert space and $A: H \to H$ a maximal monotone operator. Let $\{x_n\} \subset H$ be a sequence such that $x_n \rightharpoonup x$, $Ax_n \rightharpoonup y$, and $\limsup (x_n, Ax_n) \leq (x, y)$. Then $Ax = y$ and $\underset{n}{\lim} (x_n, Ax_n) = (x,Ax)$.
    \label{Lemmabres}
\end{lemma}

Next, we establish that the specific nonlinear functions appearing in our problem satisfy these properties.

\begin{lemma}
    The functions $\theta, \Phi: L^2(\Omega) \to L^2(\Omega)$ are maximal monotone.
    \label{maximalmonotone}
\end{lemma}

\begin{proof}
    Let $g \in \{\theta, \Phi\}$. By assumptions \ref{H1} and \ref{H3}, the function $g: \mathbb{R} \to \mathbb{R}$ is Lipschitz continuous and non-decreasing. Consequently, the induced operator $g: L^2(\Omega) \to L^2(\Omega)$ is well-defined and monotone.
    
    To prove maximality, we must show that the operator $G = I + g$ maps $L^2(\Omega)$ onto itself. Consider the scalar function $G(s) = s + g(s)$ for $s \in \mathbb{R}$. Since $g$ is non-decreasing, we have $G'(s) = 1 + g'(s) \geq 1$. This lower bound implies that $G$ is strictly increasing and coercive ($|G(s)| \to \infty$ as $|s| \to \infty$), ensuring that $G: \mathbb{R} \to \mathbb{R}$ is a bijection.
    
    Furthermore, since $G' \geq 1$, the inverse function $G^{-1}$ is Lipschitz continuous. Therefore, the operator $G$ induces a bijection on $L^2(\Omega)$, implying that $\operatorname{range}(I + g) = L^2(\Omega)$. Thus, $g$ is maximal monotone.
\end{proof}

We now provide the proof of Theorem \ref{eistenceforsemidis}.

\begin{proof}
    By Proposition \ref{prop:existofreg}, for each $\epsilon > 0$, there exists a unique solution $u_{n,\epsilon}$ to \eqref{regsemidesc}, satisfying uniform bound $\|u_{n,\epsilon}\|_V \leq C$. Since $V$ is a Hilbert space, there exists $u_n \in V$ such that, up to a subsequence, $u_{n,\epsilon} \rightharpoonup u_n$ weakly in $V$ as $\epsilon \to 0$. This implies
    \begin{equation}
        (u_{n,\epsilon}, v) \to (u_n, v), \quad \forall v \in L^2(\Omega),
    \end{equation}
    and
    \begin{equation}
        (u_{n,\epsilon}, v)_\omega \to (u_n, v)_\omega, \quad \forall v \in V.
    \end{equation}
    Since $\theta$ is Lipschitz continuous, we have
    \begin{equation}
        \|\theta(u_{n,\epsilon})\| \leq L_\theta \|u_{n,\epsilon}\| \leq L_\theta C.
    \end{equation}
    Thus, $\{\theta(u_{n,\epsilon})\}_{\epsilon > 0}$ is bounded in $L^2(\Omega)$, and up to a subsequence, there exists $w \in L^2(\Omega)$ such that $\theta(u_{n,\epsilon}) \rightharpoonup w$. Additionally, since $\{u_{n,\epsilon}\}$ is bounded, $\epsilon u_{n,\epsilon} \to 0$ strongly in $L^2(\Omega)$. Passing to the limit $\epsilon \to 0$ in \eqref{regsemidesc}, we obtain
    \begin{equation}
        (w, v) + \tau (u_n, v)_\omega = \langle f, v \rangle, \quad \forall v \in V.
        \label{weakconveq}
    \end{equation}
    By Lemma \ref{maximalmonotone}, $\theta$ is a maximal monotone operator. By Lemma \ref{Lemmabres}, it suffices to show that $\limsup_{\epsilon \to 0} (u_{n,\epsilon}, \theta(u_{n,\epsilon})) \leq (u_n, w)$. Testing \eqref{regsemidesc} with $v = u_{n,\epsilon}$ yields
    \begin{equation}
        (\theta(u_{n,\epsilon}), u_{n,\epsilon}) = \langle f, u_{n,\epsilon} \rangle - \tau |u_{n,\epsilon}|_V^2 - \epsilon \|u_{n,\epsilon}\|^2.
    \end{equation}
    Taking the $\limsup$ as $\epsilon \to 0$:
    \begin{equation}
        \limsup_{\epsilon \to 0} (\theta(u_{n,\epsilon}), u_{n,\epsilon}) \leq \langle f, u_n \rangle - \tau \liminf_{\epsilon \to 0} |u_{n,\epsilon}|_V^2.
    \end{equation}
    Since the semi-norm $|\cdot|_V$ is convex and continuous, it is weakly lower semi-continuous (see \cite{Brezis2011}, Corollary 3.9), so $\liminf |u_{n,\epsilon}|_V^2 \geq |u_n|_V^2$. Testing \eqref{weakconveq} with $v = u_n$ we obtain $(w, u_n) = \langle f, u_n \rangle - \tau |u_n|_V^2$. Therefore,
    \begin{equation}
        \limsup_{\epsilon \to 0} (\theta(u_{n,\epsilon}), u_{n,\epsilon}) \leq \langle f, u_n \rangle - \tau |u_n|_V^2 = (w, u_n).
    \end{equation}
    We conclude that $w = \theta(u_n)$, and thus equation \eqref{semidesceq} has at least one solution.

    To prove uniqueness, suppose there exist two solutions $u^{(1)}$ and $u^{(2)}$, where we drop the index $n$. Taking the difference and testing with $v = u^{(1)} - u^{(2)} \in V$, we get
    \begin{equation}
        (\theta(u^{(1)}) - \theta(u^{(2)}), u^{(1)} - u^{(2)}) + \tau |u^{(1)} - u^{(2)}|_V^2 = 0.
    \end{equation}
    Since both terms are non-negative (due to the monotonicity of $\theta$), we must have
    \begin{equation}
        (\theta(u^{(1)}) - \theta(u^{(2)}))(u^{(1)} - u^{(2)}) = 0 \quad \text{a.e. in } \Omega.
        \label{uniqdiff}
    \end{equation}
    Let $S = \{\textbf{x} \in \Omega : u^{(1)}(\textbf{x}) \neq u^{(2)}(\textbf{x})\}$. From \eqref{uniqdiff}, we have $\theta(u^{(1)}) = \theta(u^{(2)})$ a.e. on $S$. Since $\theta$ is strictly increasing, this implies $u^{(1)} = u^{(2)}$ a.e. on $S$. This is a contradiction unless $S$ has measure zero. Thus, $u^{(1)} = u^{(2)}$ a.e.
\end{proof}

\section{Proof of Theorem \ref{theorem1}}
This section presents the proof of Theorem \ref{theorem1}. We rely on Theorem \ref{eistenceforsemidis} to establish existence for the discrete problem, and subsequently pass to the limit as $\tau \to 0$ to show that the regularized continuous problem possesses a solution. Finally, we pass to the limit as $\epsilon \to 0$ to obtain the weak solution to the original problem.

For $\epsilon > 0$, we regularize $K$ by replacing it with $K_\epsilon = K + \epsilon$. We seek $u^\epsilon \in X\cap L^\infty(I; L^2(\Omega))$ satisfying $\theta(u^\epsilon(0))=\theta(u_0) \in L^2(\Omega)$, $\theta(u^\epsilon) \in X\cap L^\infty(I; L^2(\Omega))$, and $\frac{\partial \theta(u^\epsilon)}{\partial t} \in X'$, such that
\begin{equation}
    \begin{aligned}
        &\int_I \bigg\langle \frac{\partial \theta(u^\epsilon)}{\partial t}, v\bigg\rangle_{\mathcal{H}'\times \mathcal{H}} dt + \int_I ((K(u^\epsilon) + \epsilon)\nabla u^\epsilon, \nabla v)dt\\ &+ \int_I \bigg(\bar K(u^\epsilon), \nabla v \bigg) dt = \int_I (\mathcal{S}(u^\epsilon), v)dt, \forall v \in X.
    \end{aligned}
    \label{continuicereg}
\end{equation}
\begin{remark}
    By the Lions–Magenes lemma \cite{LionsMagenes1972}, the regularity specified implies that $\theta(u_\epsilon)\in C(I;L^2(\Omega))$.
\end{remark}
Using Theorem \ref{eistenceforsemidis}, we establish the following estimates.
\begin{proposition}
    The following estimate holds:
    \begin{equation}
        \sup_n \| \theta(u_n^\epsilon)\| +\tau\sum_{k=1}^N \|\nabla u_{n}^\epsilon\|^2 + \sum_{k=1}^N \|\theta(u_n^\epsilon) - \theta(u_{n-1}^\epsilon)\|^2 \leq C(\epsilon),
    \end{equation}
    where $u_n^\epsilon$ is the solution of 
    \begin{equation}
        \left(\theta(u_n^\epsilon) - \theta(u_{n-1}^\epsilon), v\right) + \tau((K(u_{n-1}^\epsilon)+\epsilon) \nabla u_{n}^\epsilon, \nabla v)+  \tau \bigg(\bar K(u_{n-1}^\epsilon), \nabla v \bigg) = \tau (\mathcal{S}(u^\epsilon_{n-1}),v),\; \forall v\in \mathcal{H},
        \label{regulirisedsemidisc}
    \end{equation}
    and $C(\epsilon)>0$ depends only on $\epsilon$.
    \label{proptimedis}
\end{proposition}
\begin{proof}
    Let $u_n^\epsilon \in \mathcal H$ be the solution of \eqref{regulirisedsemidisc}.
    Testing this equation with $v=u_{k}^\epsilon$ for $1\leq k \leq N$ yields:
    \begin{equation}
        (\theta(u_k^\epsilon) - \theta(u_{k-1}^\epsilon), u_{k}^\epsilon) + \tau ((K(u_{k-1}^\epsilon)+\epsilon) \nabla u_{k}^\epsilon, \nabla u_{k}^\epsilon) + \tau \bigg(\bar K(u_{k-1}^\epsilon), \nabla u_{k}^\epsilon\bigg) = \tau (\mathcal{S}(u^\epsilon_{k-1}),u^\epsilon_k).
    \end{equation}
    Summing over $k=1, \dots, N$, we obtain
    \begin{equation}
        \begin{aligned}
            &\sum_{k=1}^N (\theta(u_k^\epsilon) - \theta(u_{k-1}^\epsilon), u_{k}^\epsilon) + \tau \sum_{k=1}^N  ((K(u_{k-1}^\epsilon)+\epsilon) \nabla u_{k}^\epsilon, \nabla u_{k}^\epsilon) \\&+ \tau \sum_{k=1}^N \left\{\bigg(\bar K(u_{k-1}^\epsilon), \nabla u_{k-1}^\epsilon\bigg) - (\mathcal{S}(u^\epsilon_{k-1}),u^\epsilon_{k})\right\}= 0. 
        \end{aligned}   
    \end{equation}
    This equation can be decomposed as
    \begin{equation}
        T_1 + T_2 + T_3 =0.
    \end{equation}
    Regarding $T_1$, we use the fact that $\theta' \geq 0$. For all $x,y \in \mathbb R$, the inequality holds $$(\theta(x)-\theta(y))x \geq \int_y^x \theta'(s) sds.$$ Consequently,
    \begin{align*}
        T_1 &\geq \int_\Omega \sum_{k=1}^N \int_{u_{k-1}^\epsilon}^{u_{k}^\epsilon} \theta'(s) s ds d\textbf{x}\\
        & =  \int_\Omega  \int_{u_{0}^\epsilon}^{u_{N}^\epsilon} \theta'(s) s ds d\textbf{x}\\
        & =  \int_\Omega  \underbrace{\int_{0}^{u_{N}^\epsilon} \theta'(s) s ds}_{\geq 0} - \int_{0}^{u_{0}^\epsilon} \theta'(s) s ds d\textbf{x}\\
        & \geq  -\int_\Omega\int_{0}^{u_{0}^\epsilon} \theta'(s) s ds d\textbf{x} \\
        & \geq  -L_\theta\int_\Omega \int_{0}^{u_{0}^\epsilon} s ds d\textbf{x} \\
         & = -\frac{L_\theta}{2}\|u_0\|^2.
    \end{align*}
    For $T_2$, we observe that
    \begin{equation}
        T_2 \geq \epsilon \tau\sum_{k=1}^N \|\nabla u_n^\epsilon\|^2.
    \end{equation}
    Finally, for $T_3$, Cauchy-Schwarz's, Young's and Poincaré inequalities yield
    \begin{align*}
        -T_3 &\leq \tau \sum_{k=1}^N  \qty{\bigg|(\bar K(u_{k-1}^\epsilon), \nabla u_{k}^\epsilon)\bigg|+\left|(\mathcal S(u_{k-1}^\epsilon),u_k^\epsilon)\right|}\\
        &\leq \tau \sum_{k=1}^N  |\Omega|^{\frac 12}\qty{M\|\nabla u_n^\epsilon\|+M_S \|u_k^\epsilon\|}\\
        & \leq  |\Omega|^{\frac 12}(M +C_P M_S) \tau \sum_{k=1}^N\|\nabla u_k^\epsilon\|\\
        &\leq\frac{|\Omega|T^2(M +C_P M_S)^2}{2\epsilon} + \frac{\epsilon\tau}{2}\sum_{k=1}^N \|\nabla u_k^\epsilon\|^2, \text{ for } \tau \leq T.
    \end{align*} 
    Combining the above estimates results in
    \begin{equation}
        \tau\sum_{k=1}^N \|\nabla u_n^\epsilon\|^2 \leq\frac{ |\Omega|T^2(M +C_P M_S)^2}{\epsilon^2} + \frac{L_\theta}{\epsilon}\|u_0\|^2 =: C_1(\epsilon).
    \end{equation}
    Testing with $v = \theta(u_n^\epsilon)$ is possible because $\theta$ is Lipschitz and $\theta(0) = 0$. Summing from $k=1$ to $n$ (where $1\leq n \leq N$), we obtain 
    \begin{equation}
        \begin{aligned}
            &\sum_{k=1}^n (\theta(u_k^\epsilon) - \theta(u_{k-1}^\epsilon), \theta(u_k^\epsilon)) + \tau \underbrace{\sum_{k=1}^n  ((K(u_{k-1}^\epsilon)+\epsilon)\theta'(u_k^\epsilon) \nabla u_{k}^\epsilon, \nabla u_{k}^\epsilon)}_{\geq 0}\\ &+ \tau \sum_{k=1}^n \left\{\bigg(\bar K(u_{k-1}^\epsilon)\theta'(u_k^\epsilon), \nabla u_{k}^\epsilon\bigg)-(\mathcal S(u_{k-1}^\epsilon),\theta(u_k)) \right\}= 0.
        \end{aligned}   
    \end{equation}
    Denoting this as $T_4+T_5+T_6 = 0$, we note that $T_5 \geq 0$. Similar to the bound for $T_3$ with the fact that $\theta$ is Lipschitz, we find
    \begin{equation}
        -T_6 \leq \frac{ |\Omega|T(M+C_PM_S)^2 L_\theta ^2}{2} + \frac{\tau}{2}\sum_{k=1}^N \|\nabla u_k^\epsilon\|^2 \leq \frac{ |\Omega|T(M+C_PM_S)^2 L_\theta ^2}{2} + \frac{C_1(\epsilon)}{2}.
    \end{equation}
    For $T_4$, algebraic identity
    \begin{equation}
        (x-y)x = \frac{1}{2}(x^2 -y^2 + (x-y)^2)
    \end{equation}
    yields
    \begin{align*}
        T_4 &= \frac{1}{2}\sum_{k=1}^n \left\{\|\theta(u_k^\epsilon)\|^2 - \|\theta(u_{k-1}^\epsilon)\|^2 + \|\theta(u_k^\epsilon)-\theta(u_{k-1}^\epsilon)\|^2\right\}\\ 
        &= \frac{1}{2}\|\theta(u_{n}^\epsilon)\|^2 - \frac{1}{2}\|\theta(u_{0})\|^2 + \frac{1}{2}\sum_{k=1}^n\|\theta(u_k^\epsilon)-\theta(u_{k-1}^\epsilon)\|^2.
    \end{align*}
    We conclude that
    \begin{equation}
        \sup_n \|\theta(u_n^\epsilon) \|^2 + \sum_{k=1}^N \|\theta(u_k^\epsilon)-\theta(u_{k-1}^\epsilon)\|^2 \leq |\Omega|T(M+C_PM_S)^2L_\theta^2 + \|\theta(u_0)\|^2 + C_1(\epsilon) =: C_2(\epsilon).
    \end{equation}
\end{proof}
\begin{corollary}
    Under \ref{H2}, $\alpha |\zeta| \leq |\theta(\zeta)|$ for all $\zeta\in \mathbb R$, we have that $\underset{n}{\sup} \|u_n^\epsilon\| \leq C(\epsilon)$.
    \label{Cortimedisc}
\end{corollary}
\begin{theorem}
    The regularized equation \eqref{continuicereg} possesses a solution.
    \label{regsolconvintime}
\end{theorem}
\begin{proof}
    We define the interpolant
    \begin{equation}
        u^{\tau, \epsilon}(t) = u_n^\epsilon, \text{ for } t_{n-1} < t \leq t_n.
    \end{equation}
    By Proposition \ref{proptimedis}, $u^{\tau,\epsilon}$ is uniformly bounded (with respect to $\tau$) in $L^\infty(I; L^2(\Omega)) \cap X$. Defining the piecewise linear interpolant
    \begin{equation}
        \theta^{lin,\tau, \epsilon}(t) =  \frac{t-t_{n-1}}{\tau}(\theta(u_n^\epsilon) - \theta(u_{n-1}^\epsilon)) + \theta(u_{n-1}^\epsilon), \text{ for } t_{n-1} \leq t \leq t_n,
    \end{equation}
    it follows that
    \begin{equation}
        \frac{\partial \theta^{lin,\tau, \epsilon}}{\partial t} = \frac{\theta(u_n^\epsilon) - \theta(u_{n-1}^\epsilon)}{\tau}, \text{ for } t_{n-1}< t < t_n.
    \end{equation}
    Proposition \ref{proptimedis} and the fact that $u_n^\epsilon$ solves \eqref{regulirisedsemidisc} imply that $\frac{\partial \theta^{lin,\tau, \epsilon}}{\partial t}$ is uniformly bounded in $L^2(I, \mathcal{H}')$. Indeed, \eqref{regulirisedsemidisc} gives
    \begin{equation}
        \bigg(\frac{\theta(u_n^\epsilon) - \theta(u_{n-1}^\epsilon)}{\tau} , v\bigg) = -(K_\epsilon(u_{n-1}^\epsilon)\nabla u_n^\epsilon, \nabla v)-\bigg(\bar K(u_{n-1}^\epsilon), \nabla v \bigg)+\qty(\mathcal S(u_{n-1}^\epsilon),v), \forall v \in \mathcal H.
    \end{equation}
    Therefore,
    \begin{equation}
        \begin{aligned}
        \frac{|(\frac{\theta(u_n^\epsilon) - \theta(u_{n-1}^\epsilon)}{\tau} , v)|}{\|v\|_1} &\leq \frac{|(K_\epsilon(u_{n-1}^\epsilon)\nabla u_n^\epsilon, \nabla v)|}{\|v\|_1} +\frac{|(\bar K(u_{n-1}^\epsilon), \nabla v )|}{\|v\|_1}+\frac{\qty|\qty(\mathcal{S}(u_{n-1}^\epsilon),v)|}{\|v\|_1} \\&\leq |\Omega|^{\frac 12}(M+1) \|\nabla u_n^\epsilon\| + |\Omega|^{\frac 12}(M + C_PM_S), \text{ for } \epsilon\leq 1.
    \end{aligned}
    \end{equation}
    Consequently,
    \begin{equation}
        \left\|\frac{\theta(u_n^\epsilon) - \theta(u_{n-1}^\epsilon)}{\tau}\right\|_{-1} \leq |\Omega|^{\frac 12}(M+1) \|\nabla u_n^\epsilon\| +|\Omega|^{\frac 12}( M+C_PM_S),
    \end{equation}
    which implies
    \begin{equation}
        \left\|\frac{\theta(u_n^\epsilon) - \theta(u_{n-1}^\epsilon)}{\tau}\right\|_{-1}^2 \leq 2||\Omega|(M+1)^2 \|\nabla u_n^\epsilon\|^2 + 2|\Omega|(M+C_PM_S)^2.
    \end{equation}
    Summing over $n$ yields
    \begin{equation}
        \tau \sum_{n=1}^N \left\|\frac{\theta(u_n^\epsilon) - \theta(u_{n-1}^\epsilon)}{\tau}\right\|_{-1}^2 \leq 2T|\Omega|(M+C_PM_S)^2+ 2|\Omega|(M+1)^2 \tau \sum_{n=1}^N\|\nabla u_n^\epsilon\|^2  \leq C(\epsilon).
    \end{equation}
    Then proposition \ref{proptimedis} ensures that $\theta^{lin,\tau, \epsilon}$ is uniformly bounded in $X$. Observing that
    \begin{equation}
        | \theta^{lin,\tau, \epsilon}(t)| = \left|\frac{t-t_{n-1}}{\tau}(\theta(u_n^\epsilon) - \theta(u_{n-1}^\epsilon)) + \theta(u_{n-1}^\epsilon)\right|\leq |\theta(u_n^\epsilon)| + |\theta(u_{n-1}^\epsilon)| \text{ for } t_{n-1} \leq t \leq t_n,
    \end{equation}
    and similarly for the gradient:
    \begin{equation}
        |\nabla \theta^{lin,\tau, \epsilon}(t)| = \left|\frac{t-t_{n-1}}{\tau}(\nabla\theta(u_n^\epsilon) - \nabla\theta(u_{n-1}^\epsilon)) + \nabla\theta(u_{n-1}^\epsilon)\right|\leq |\nabla\theta(u_n^\epsilon)|+|\nabla\theta(u_{n-1}^\epsilon)|,
    \end{equation}
    we apply the Aubin–Lions lemma \cite{Aubin1963}: there exists $w^\epsilon \in L^2(I, \mathcal{H})$ with $\frac{\partial w^\epsilon}{\partial t} \in L^2(I, \mathcal{H}')$ such that
    \begin{align*}
        & \theta^{lin,\tau, \epsilon} \underset{\tau \to 0}{\to} w^\epsilon \text{ in } L^2(I; L^2(\Omega)),\\
        & \nabla \theta^{lin,\tau, \epsilon} \underset{\tau \to 0}{\rightharpoonup} \nabla w^\epsilon\text{ in } L^2(I; L^2(\Omega)),\\
        & \frac{\partial \theta^{lin,\tau, \epsilon}}{\partial t } \underset{\tau \to 0}{\rightharpoonup} \frac{\partial w^\epsilon}{\partial t }\text{ in } X'.
    \end{align*}
    Proposition \ref{proptimedis} provides the bound
    \begin{equation}
        \sum_{k=1}^N\|\theta(u_n^\epsilon)-\theta(u_{n-1}^\epsilon)\|^2 \leq C(\epsilon),
    \end{equation}
    allowing us to conclude that
    \begin{equation}
        \int_I \|\theta(u^{\tau, \epsilon}) - \theta(u^{\tau -, \epsilon})\|^2 \underset{\tau \to 0}{\to} 0,
        \label{timedelayconv}
    \end{equation}
    where $u^{\tau-, \epsilon}$ is defined by 
    \begin{equation}
        u^{\tau-, \epsilon} = u^{\tau, \epsilon}(t-\tau).
    \end{equation}
    Since $\theta^{lin,\tau, \epsilon} \underset{\tau \to 0}{\to} w^\epsilon$ and
    \begin{equation}
        \theta^{lin,\tau, \epsilon}(t) = \frac{t-t_{n-1}}{\tau}(\theta(u_n^{\tau,\epsilon}) - \theta(u_n^{\tau-,\epsilon})) + \theta(u_n^{\tau-,\epsilon}),
    \end{equation}
    it follows that
    \begin{equation}
        \theta^{lin,\tau, \epsilon}(t)-\theta(u_n^{\tau,\epsilon}) = \left(\frac{t-t_{n-1}}{\tau}-1\right)(\theta(u_n^{\tau,\epsilon}) - \theta(u_n^{\tau-,\epsilon})).
    \end{equation}
    Hence,
    \begin{equation}
        |\theta^{lin,\tau, \epsilon}(t)-\theta(u_n^{\tau,\epsilon})| \leq |\theta(u_n^{\tau,\epsilon}) - \theta(u_n^{\tau-,\epsilon})|, \text{ for } t \in [t_{n-1}, t_{n}].
    \end{equation}
    By \eqref{timedelayconv}, we obtain
    \begin{equation}
        \|\theta^{lin,\tau, \epsilon} - \theta(u^{\tau,\epsilon})\|_{L^2(I;L^2(\Omega))} \underset{\tau \to 0}{\to} 0.
    \end{equation}
    Consequently, we have in $L^2(I;L^2(\Omega))$
    \begin{align}
        &\theta(u^{\tau, \epsilon}) \underset{\tau \to 0}{\to} w^\epsilon,\\
        &\theta(u^{\tau-, \epsilon}) \underset{\tau \to 0}{\to} w^\epsilon.
        \label{thetastrongconv}
    \end{align}
    Since $u^{\tau, \epsilon}$ is uniformly bounded in $X$ by proposition \ref{proptimedis} and corollary \ref{Cortimedisc}, we have in $L^2(I; L^2(\Omega))$,
    \begin{align}
        &u^{\tau, \epsilon} \underset{\tau \to 0}{\rightharpoonup} u^\epsilon,\\
        &\nabla u^{\tau, \epsilon} \underset{\tau \to 0}{\rightharpoonup} \nabla u^\epsilon.
    \end{align}
    As $\theta$ is a maximal monotone operator on $L^2$, Lemma \ref{Lemmabres} implies that
    \begin{equation}
        w^{\epsilon}=\theta(u^\epsilon).
    \end{equation}
    Next, we verify that for all $v \in X$,
    \begin{equation}
        \int_I((K_\epsilon(u^{\tau-, \epsilon}))\nabla u^{\tau,\epsilon}, \nabla v) dt \underset{\tau \to 0}{\to} \int_I((K_\epsilon(u^{\epsilon}))\nabla u^{\epsilon}, \nabla v)dt.
    \end{equation}
    Using \ref{H1}, $u^{\tau-,\epsilon} \underset{\tau \to 0}{\to} u^\epsilon$ strongly in $L^2(\Omega \times I)$, and thus (up to a subsequence) $u^{\tau-,\epsilon} \to u^\epsilon$ a.e. Since $K_\epsilon$ is continuous and bounded, $K_\epsilon(u^{\tau-,\epsilon}) \to K_\epsilon(u^{\epsilon})$ a.e. Therefore, by the Lebesgue Dominated Convergence Theorem \cite{Royden1988},
    \[
    K_\epsilon(u^{\tau-,\epsilon}) \nabla v \underset{\tau \to 0}{\to} K_\epsilon(u^{\epsilon}) \nabla v
    \quad \text{in } L^2(\Omega \times I),
    \]
    for any $v \in X$. Combined with the weak convergence $\nabla u^{\tau,\epsilon} \underset{\tau \to 0}{\rightharpoonup} \nabla u^\epsilon$, the result follows. Similarly, continuity and boundedness of $\bar K$ and $\mathcal S$ ensures \begin{align}
        &\bar K(u^{\tau-,\epsilon}) \underset{\tau \to 0}{\to} \bar K(u^\epsilon)\\
        &\mathcal{S}(u^{\tau-,\epsilon}) \underset{\tau \to 0}{\to} \mathcal{S}(u^\epsilon)
    \end{align} in $L^2(\Omega \times I)$. 
    
    By taking the limit as $\tau \to 0$ in equation \eqref{regulirisedsemidisc}, we confirm that $u^\epsilon \in X$, with $\frac{\partial \theta(u^\epsilon)}{\partial t} \in X'$, is the solution of
    \begin{equation}
        \int _I \langle \frac{\partial \theta(u^\epsilon)}{\partial t}, v \rangle_{\mathcal H'\times \mathcal H} dt + \int_I (K_\epsilon(u^\epsilon) \nabla u^\epsilon, \nabla v) + \bigg(\bar K(u^\epsilon), \nabla v \bigg) dt = \int_I\qty(\mathcal{S}(u^\epsilon),v)dt, \forall v \in X.
    \end{equation}
    Furthermore, since $\{\theta(u^{\tau,\epsilon})\}_{\tau>0} \in L^\infty(I, L^2(\Omega))$ is uniformly bounded, thus $\theta(u^{\tau,\epsilon}) \underset{\tau \to 0}{\overset{*}{\rightharpoonup}} w_1^\epsilon$ weak-* in $L^\infty(I, L^2(\Omega))$. The strong convergence $\theta(u^{\tau,\epsilon}) \underset{\tau \to 0}{\to} \theta(u^\epsilon)$ in $L^2(I;\Omega)$ gives us that, $\theta(u^\epsilon) = w_1^\epsilon \in L^\infty(I, L^2(\Omega))$.
\end{proof}

Equation \eqref{continuicereg} can be rewritten as
\begin{equation}
    \begin{aligned}
        &\int_I\langle \frac{\partial \theta(u^\epsilon)}{\partial t}, v\rangle dt  + \int_I(\nabla \Phi(u^\epsilon), \nabla v) + \epsilon (\nabla u^\epsilon,\nabla v)  \\&+ \bigg(\bar K(u^\epsilon), \nabla v \bigg)dt =\int_I\qty(\mathcal{S}(u^\epsilon),v)dt, \forall v \in X,
    \end{aligned}
    \label{continuousregphi}
\end{equation}
with $\Phi$ defined in \eqref{defphi}.

To establish the existence of a solution to the original problem, we must pass to the limit as the regularization parameter $\epsilon \to 0$. This convergence analysis relies on compactness arguments, which require that the sequence of regularized solutions $\{u^\epsilon\}_{\epsilon > 0}$ remains bounded in specific functional spaces. The following proposition establishes these essential uniform a priori estimates, independent of $\epsilon$.
\begin{proposition}
    Let $\epsilon>0$, and let $u^\epsilon$ be a solution of \eqref{continuousregphi}. Then there exists a constant $C$ such that
    \begin{equation}
        \left\|\frac{\partial \theta(u^\epsilon)}{\partial t }\right\|_{X'} + \|\theta(u^\epsilon)\|_{L^\infty(I; L^2(\Omega))}+\|\Phi(u^\epsilon)\|_X + \|\sqrt{\epsilon} u^\epsilon\|_X \leq C.
    \end{equation}
    \label{propositionuniformbounds}
\end{proposition}
\begin{proof}
    Testing \eqref{continuousregphi} with $v=\theta(u^\epsilon) 1_{\{0<t<s\}}\in X$ yields
    \begin{equation}
        \begin{aligned}
            &\int_0^s\langle \frac{\partial \theta(u^\epsilon)}{\partial t}, \theta(u^\epsilon)\rangle dt + \int_0^s ((K(u^\epsilon)+\epsilon)\theta'(u^\epsilon)\nabla u^\epsilon, \nabla u^\epsilon)dt\\& + \int_0^s \bigg(\bar K_1(u^\epsilon), K(u^\epsilon)\theta'(u^\epsilon)\nabla u^\epsilon\bigg)-\qty(\mathcal{S}(u^\epsilon),\theta(u^\epsilon)) dt =0, \forall v \in X,.
        \end{aligned}
    \end{equation}
    This is decomposed as
    \begin{equation}
        T_1+T_2+T_3 =0.
    \end{equation}
    By Lemma 64.40 in \cite{Ern2021}, $T_1$ is written as
    \begin{equation}
        T_1 = \frac{1}{2}\|\theta(u^\epsilon(s))\|^2 - \frac{1}{2}\|\theta(u^\epsilon(0))\|^2.
    \end{equation}
    For $T_3$, Cauchy's and Young's inequalities provide
    \begin{equation}
        \begin{aligned}
            -T_3 &\leq\int_0^s ML_{\theta}^{\frac 12}|\Omega|^{\frac 12}\|\sqrt{K(u^\epsilon)\theta'(u^\epsilon)} \nabla u^\epsilon\| + M_S |\Omega|^{\frac 12}\|\theta(u^\epsilon)\|dt \\&\leq \frac{(M^2L_\theta + M_S) |\Omega|T}{2}+\frac{1}{2}\int_0^s\|\sqrt{K(u^\epsilon)\theta'(u^\epsilon)} \nabla u^\epsilon\|^2dt+\frac 12 \int_0^s \|\theta(u^\epsilon)\|^2dt.
        \end{aligned}
    \end{equation}
    Regarding $T_2$, we have
    \begin{equation}
        T_2 \geq \int_0^s\|\sqrt{K(u^\epsilon)\theta'(u^\epsilon)} \nabla u^\epsilon\|^2dt \geq 0.
    \end{equation}
    These inequalities imply
    \begin{equation}
        \|\theta(u^\epsilon(s))\|^2 \leq (M^2L_\theta + M_S)|\Omega|T + \|\theta(u_0)\|^2+\int_0^s \|\theta(u^\epsilon)\|^2dt.
    \end{equation}
    By Grönwall's inequality we have that for $t\in (0,T)$
    \begin{equation}
        \begin{aligned}
            \|\theta(u^\epsilon(t))\|^2 &\leq ((M^2L_\theta + M_S)|\Omega|T + \|\theta(u_0)\|^2)\exp(t)\\
            &\leq ((M^2L_\theta + M_S)|\Omega|T + \|\theta(u_0)\|^2)\exp(T)
        \end{aligned}
    \end{equation}
    Testing with $v=\Phi(u^\epsilon)$ gives
    \begin{equation}
        \begin{aligned}
            \int_I \langle \frac{\partial \theta(u^\epsilon)}{\partial t}, \Phi(u^\epsilon)\rangle dt &+ \int_I \|\nabla \Phi(u^\epsilon)\|^2 + \epsilon(K(u^\epsilon)\nabla u^\epsilon, \nabla u^\epsilon)dt\\ &+\int_I\bigg(\bar K(u^\epsilon), \nabla \Phi(u^\epsilon)\bigg) - \qty(\mathcal S(u^\epsilon), \Phi(u^\epsilon)) dt=0. 
        \end{aligned}
    \end{equation}
    This is rewritten as
    \begin{equation}
        T_4+T_5+T_6=0.
    \end{equation}
    Using Lemma 4 in \cite{Carrillo1999} (p.324), $T_4$ becomes
    \begin{equation}
        T_4=\int_\Omega \varphi(u^\epsilon(T)) -\varphi(u_0)d\textbf{x},
    \end{equation}
    where 
    \begin{equation}
        \varphi(s) = \int_0^s \Phi(\xi) \theta'(\xi)d\xi.
    \end{equation}
    Since $\theta' \ge 0$ and $\Phi$ has the same sign as for $\xi\in \mathbb R$, $\varphi \geq 0$. Thus, by the first part of Lemma 4 in \cite{Carrillo1999} (p.324):
    \begin{equation}
        T_4 \ge -\int_\Omega\varphi(u(0))d\textbf{x}=:-C_0 > -\infty.
    \end{equation}
    For $T_5$, we observe that
    \begin{equation}
        T_5 \geq \|\Phi(u^\epsilon)\|^2_X,
    \end{equation}
    and for $T_6$,
    \begin{equation}
        -T_6 \leq \frac{(M+C_PM_S)^2|\Omega|T}{2} + \frac{1}{2}\|\Phi(u^\epsilon)\|^2_X.
    \end{equation}
    Therefore,
    \begin{equation}
        \|\Phi(u^\epsilon)\|^2_X \leq(M+C_PM_S)^2|\Omega|T + 2C_0.
    \end{equation}
    Testing the equation with $v=u^\epsilon$, the first term is treated via Lemma 1.5 in \cite{AltLuckhaus1983} as
    \begin{equation}
        \int_I \langle \frac{\partial \theta(u^\epsilon)}{\partial t}, u^\epsilon\rangle dt = \int_\Omega \Theta(u^\epsilon(T)) - \Theta(u(0)) d\textbf{x} \geq -\int_\Omega  \Theta(u_0) d\textbf{x}=-C_1,
    \end{equation}
    where
    \begin{equation}
        \Theta(s) = \int_0^s \xi \theta'(\xi) d\xi \ge 0.
    \end{equation}
    The second term satisfies
    \begin{equation}
        \int_I (K(u^\epsilon) \nabla u^\epsilon, \nabla u^\epsilon) + \epsilon\|\nabla u^\epsilon\|^2 dt \geq \epsilon\| u^\epsilon\|^2_X,
    \end{equation}
    and the last term is bounded by
    \begin{equation}
        \int_I -\bigg(\bar K_1(u^\epsilon), {K(u^\epsilon)}\nabla u^\epsilon\bigg) + \qty(\mathcal S(u^\epsilon), u^\epsilon)dt \leq |\Omega|^{\frac 12}(M \|\Phi(u^\epsilon)\|_X +M_S\|u^\epsilon\|) \leq C_3.
    \end{equation}
    Consequently,
    \begin{equation}
         \|\sqrt{\epsilon}u^\epsilon\|_X \leq C_4.
    \end{equation}
    We have
    \begin{equation}
        \begin{aligned}
            \int_I \bigg\langle \frac{\partial\theta(u^\epsilon)}{\partial t}, v\bigg\rangle dt = &\int_I-(\nabla \Phi(u^\epsilon), \nabla v) - \sqrt{\epsilon} (\sqrt{\epsilon}\nabla u^\epsilon,\nabla v) \\& - \bigg(\bar K(u^\epsilon), \nabla v \bigg)+\qty(\mathcal{S}(u^\epsilon),v)dt, \forall v \in X.
        \end{aligned}
    \end{equation}
    Taking the absolute value, dividing by $\|v\|_X$, and taking the supremum over $v$ yields:
    \begin{equation}
        \left\|\frac{\partial\theta(u^\epsilon)}{\partial t}\right\|_{X'} \leq \|\Phi (u^\epsilon) \| + \sqrt{\epsilon}\|\sqrt{\epsilon}\nabla u^\epsilon \| + \|\bar K(u^\epsilon)\|+\|\mathcal{S}(u^\epsilon)\| \leq C \text{ for } \epsilon \to 0. 
    \end{equation}
\end{proof}
In the next step, we use the bounds of Proposition \ref{propositionuniformbounds} to extract convergent subsequences and pass to the limit $\epsilon\to 0$, thereby proving Theorem \ref{theorem1}.

\begin{proof}
    Proof of Theorem \ref{theorem1}.
    
    Proposition \ref{propositionuniformbounds} ensures that $\theta(u^\epsilon)$ and $u^\epsilon$ are uniformly bounded in $L^\infty(I, L^2(\Omega)) \subset L^2(I; L^2(\Omega))$, $\frac{\partial \theta(u^\epsilon)}{\partial t}$ is uniformly bounded in $X'$, and $\Phi(u^\epsilon)$ is uniformly bounded in $X$. Applying Theorem 1 in \cite{Moussa2016}, we extract a subsequence such that
    \begin{align}
        \theta(u^\epsilon) &\underset{\epsilon \to 0}{\to} w \quad \text{strongly in } L^2(I; L^2(\Omega)), \\
        \frac{\partial \theta(u^\epsilon)}{\partial t} &\underset{\epsilon \to 0}{\rightharpoonup} \frac{\partial w}{\partial t} \quad \text{weakly in } X',
    \end{align}
    and 
    \begin{align}
        \Phi(u^\epsilon) &\underset{\epsilon \to 0}{\rightharpoonup} v \quad \text{weakly in } L^2(I; L^2(\Omega)),\\
        \nabla \Phi(u^\epsilon) &\underset{\epsilon \to 0}{\rightharpoonup} \nabla v \quad \text{weakly in } L^2(I; L^2(\Omega)).
    \end{align}
    By Proposition \ref{propositionuniformbounds}, we have $u^\epsilon \underset{\epsilon \to 0}{\rightharpoonup} u$ weakly in $L^2(I;L^2(\Omega))$. Consequently, Lemma~\ref{Lemmabres} implies that $w = \theta(u)$. Using Assumption~\ref{H1}, we obtain the strong convergence $u^\epsilon \underset{\epsilon \to 0}{\to} u$ in $L^2(I; L^2(\Omega))$, which implies convergence almost everywhere in $\Omega \times I$ (up to a subsequence). Since $\bar K$ and $\mathcal S$ are bounded and continuous, the Lebesgue Dominated Convergence Theorem yields
\begin{equation}
    \begin{aligned}
        &\bar K(u^\epsilon) \underset{\epsilon \to 0}{\to} \bar K(u),\\
        &\mathcal S(u^\epsilon) \underset{\epsilon \to 0}{\to} \mathcal S(u)
    \end{aligned}
\end{equation}
in $L^2(\Omega \times I)$. Applying Lemma~\ref{Lemmabres} again allows us to identify $v = \Phi(u)$.

Furthermore, observing that $\sqrt{\epsilon}(\sqrt{\epsilon}\nabla u^\epsilon) \underset{\epsilon \to 0}{\to} 0$, we pass to the limit as $\epsilon \to 0$ in \eqref{continuousregphi} to recover \eqref{uform}. Finally, since $\theta(u^\epsilon)$ and $u^\epsilon$ are uniformly bounded in $L^{\infty}(I;L^2(\Omega))$, we conclude that
\[
    \theta(u),\, u \in L^{\infty}(I; L^2(\Omega)).
\]
\end{proof}
We also demonstrate that $\theta(u) \in C_w( I; L^2(\Omega))$, where $C_w( I; L^2(\Omega))$ denotes the space of functions $u: t\mapsto u(t) \in L^2(\Omega)$ such that for all $v\in L^2(\Omega)$, the map $t \mapsto (u(t),v) \in C(\bar I)$.

\begin{proposition}
    If $y \in L^\infty(I; L^2(\Omega))$ and $\frac{\partial y}{\partial t} \in X'$, then $y \in C_w(I; L^2(\Omega))$.
    \label{weakcontinuity}
\end{proposition}
\begin{proof}
    Let $v \in \mathcal{H}$. We define the scalar function
    \begin{equation}
        f_v(t) = (y(t), v).
    \end{equation}
    According to Proposition 64.33 in \cite{Ern2021}, the weak derivative is given by
    \begin{equation}
        f_v'(t) = \bigg\langle \frac{\partial y}{\partial t}(t), v \bigg\rangle_{\mathcal{H}'\times \mathcal{H}}.
    \end{equation}
    We observe the following bounds:
    \begin{equation}
        \|f_v\|_{L^2(I)} = \sqrt{\int_I |(y(t), v)|^2 dt} \leq \|y\|_{L^2(I; L^2(\Omega))} \|v\| < \infty, 
    \end{equation}
    and
    \begin{equation}
        \|f_v'\|_{L^2(I)} = \sqrt{\int_I |f'_v(t)|^2 dt} \leq \left\|\frac{\partial y}{\partial t}\right\|_{X'} \|v\|_1 < \infty.
    \end{equation}
    These estimates imply that $f_v \in H^1(I)$. By the Sobolev embedding theorem in one dimension, $H^1(I) \hookrightarrow C(\overline{I})$, so $f_v$ is continuous.
    
    To extend this to $v \in L^2(\Omega)$, we use a density argument. Since $\mathcal{H}$ is dense in $L^2(\Omega)$, there exists a sequence $\{v_n\} \subset \mathcal{H}$ such that $v_n \to v$ in $L^2(\Omega)$. By the Cauchy-Schwarz inequality, we have
    \begin{equation}
        |f_{v_n}(t) - f_v(t)| = |(y(t), v_n - v)| \leq \|y(t)\| \|v_n - v\|.
    \end{equation}
    Since $y \in L^\infty(I; L^2(\Omega))$, we can take the supremum over time:
    \begin{equation}
        \|f_{v_n} - f_v\|_{C(\overline{I})} \leq \|y\|_{L^\infty(I; L^2(\Omega))} \|v_n - v\|.
    \end{equation}
    As $n \to \infty$, the right-hand side vanishes. Thus, $f_{v_n}$ converges uniformly to $f_v$. Since the uniform limit of continuous functions is continuous, $f_v \in C(\overline{I})$. Consequently, the map $t \mapsto (y(t), v)$ is continuous for all $v \in L^2(\Omega)$, which means $y \in C_w(I; L^2(\Omega))$.
\end{proof}
By Proposition \ref{weakcontinuity}, $\theta(u)\in C_w(I;L^2(\Omega))$, ensuring the trace is well defined and satisfying
\begin{equation}
    \int_0^s \bigg\langle \frac{\partial \theta(u)}{\partial t}, v\bigg\rangle dt = (\theta(u(s)),v) - (\theta(u_0),v), \forall v \in \mathcal{H}.
\end{equation}

\section{Maximum Principle}

In addition to existence and uniqueness, it is crucial to ensure that the mathematical model preserves the physical bounds of the saturation variable. In this section, we establish an $L^\infty$-estimate for the weak solution, guaranteeing that the solution remains bounded throughout the time evolution provided the initial and boundary data is bounded.

\begin{theorem}[Generalized Maximum Principle]
Let $u$ be a weak solution to the Richards equation \eqref{uform} with initial data $u_0 \in L^\infty(\Omega)$ and boundary data $u_d \in L^\infty(\partial\Omega \times I)$. Assume that:
\begin{enumerate}
    \item The source term $\mathcal{S}$ satisfies hypothesis \ref{H4}, specifically $|\mathcal{S}(\textbf{x},t,u)| \le M_{\mathcal{S}}$ almost everywhere.
    \item The convective coefficient $\bar{K}(\textbf{x},t,u)$ has a bounded spatial divergence such that $|\nabla_{\textbf{x}} \cdot \bar{K}(\textbf{x},t,u)| \le M_K$ almost everywhere for some constant $M_K \ge 0$.
\end{enumerate}
Define the bounds on the initial and boundary conditions:
\begin{equation}
    M_0 = \max\left( \operatorname*{ess\,sup}_{\textbf{x} \in \Omega} u_0(\textbf{x}), \ \operatorname*{ess\,sup}_{(\textbf{x},t) \in \partial \Omega \times I} u_d(\textbf{x},t) \right),
\end{equation}
\begin{equation}
    m_0 = \min\left( \operatorname*{ess\,inf}_{\textbf{x} \in \Omega} u_0(\textbf{x}), \ \operatorname*{ess\,inf}_{(\textbf{x},t) \in \partial \Omega \times I} u_d(\textbf{x},t) \right).
\end{equation}
Let $M = M_{\mathcal{S}} + M_K$. Then, the solution satisfies the following global estimate:
\begin{equation}
    \theta^{-1}(\theta(m_0) - Mt) \le u(\textbf{x},t) \le \theta^{-1}(\theta(M_0) + Mt) \quad \text{for a.e. } (\textbf{x},t) \in \Omega \times I.
\end{equation}
\label{maxprincipe}
\end{theorem}

\begin{proof}
    \textbf{Step 1: Upper Bound.}
    Let $Y(t)$ be the spatially constant supersolution defined by the ODE:
    \begin{equation}
        \frac{\partial \theta(Y(t))}{\partial t} = M,
    \end{equation}
    where the solution is given by $\theta(Y(t)) = \theta(M_0) + M t$.
    Since $\theta$ is strictly increasing, $Y(t)$ is well-defined, and $Y(0) = M_0$.
    
    We define the test function $v$ explicitly as the non-negative part of the difference in $\Phi$:
    \begin{equation}
        v = [\Phi(u) - \Phi(Y(t))]_+ = \max(0, \Phi(u) - \Phi(Y(t))).
    \end{equation}
    Since $Y(t) \ge M_0 \ge u_d$ on the boundary $\partial \Omega$, we have $\Phi(u) \le \Phi(Y)$ on $\partial \Omega$, implying $v = 0$ on the boundary and thus $v \in L^2(I; H_0^1(\Omega))$.
    
    Subtracting the equation for $Y(t)$ (where $\partial_t \theta(Y) = M$ and $\nabla Y = 0$) from the weak formulation of \eqref{mainequation}, and testing with $v$, we obtain:
    \begin{equation}
        \int_0^t \left\langle \frac{\partial}{\partial t}(\theta(u) - \theta(Y)), v \right\rangle ds + \int_0^t (\nabla \Phi(u), \nabla v) ds + \int_0^t (\bar{K}(\textbf{x}, t, u), \nabla v) ds = \int_0^t (\mathcal{S} - M, v) ds.
    \end{equation}
    We analyze the terms $I_1 + I_2 + I_3 = I_4$.
    
    \textbf{Time Term $I_1$:}
    Using the monotonicity of $\theta$ and $\Phi$, the time derivative term satisfies (see Lemma 4 in \cite{Carrillo1999}):
    \begin{equation}
        I_1 = \int_\Omega \mathcal{E}(u(t), Y(t)) d\textbf{x} - \int_\Omega \mathcal{E}(u_0, Y(0)) d\textbf{x},
    \end{equation}
    where the entropy function $\mathcal{E}$ is defined explicitly as:
    \begin{equation}
        \mathcal{E}(u, Y) = \int_Y^u \theta'(\xi) [\Phi(\xi) - \Phi(Y)]_+ d\xi = 
        \begin{cases}
            \int_Y^u \theta'(\xi)(\Phi(\xi) - \Phi(Y)) d\xi & \text{if } u > Y, \\
            0 & \text{if } u \le Y.
        \end{cases}
    \end{equation}
    Since $\mathcal{E}(u, Y) \ge 0$ is a function that vanishes if and only if $u \le Y$, and $u_0 \le M_0 = Y(0)$, the initial term vanishes, yielding $I_1 \ge 0$.
    
    \textbf{Diffusion Term $I_2$:}
    Since $Y$ is spatially constant, $\nabla \Phi(Y) = 0$. Thus:
    \begin{equation}
        I_2 = \int_0^t \int_{\{u > Y\}} |\nabla \Phi(u)|^2 d\textbf{x} ds \ge 0.
    \end{equation}
    
    \textbf{Convection Term $I_3$:}
    We define the primitive $\mathcal{K}$ with respect to the moving bound $Y(t)$:
    \begin{equation}
        \mathcal{K}(\textbf{x}, t, s) = 
        \begin{cases}
            \int_{Y(t)}^s \bar{K}(\textbf{x}, t, \xi) \Phi'(\xi) d\xi & \text{if } s > Y(t), \\
            0 & \text{if } s \le Y(t).
        \end{cases}
    \end{equation}
    By applying the chain rule, the total spatial divergence of this primitive is $\nabla \cdot \mathcal{K}(\textbf{x}, t, u) = \nabla_{\textbf{x}} \cdot \mathcal{K}(\textbf{x}, t, u) + \bar{K}(\textbf{x}, t, u) \cdot \nabla \Phi(u)$ on the set where $v>0$, where $\nabla_{\textbf{x}} \cdot$ denotes the partial divergence with respect to the explicit spatial coordinates $\textbf{x}$ (treating $u$ as an independent variable). Applying Green's theorem and noting $v=0$ on $\partial \Omega$:
    \begin{equation}
        I_3 = - \int_0^t \int_\Omega \nabla_{\textbf{x}} \cdot \mathcal{K} \, d\textbf{x} ds = - \int_0^t \int_\Omega \int_{Y(t)}^u \nabla_{\textbf{x}} \cdot \bar{K}(\textbf{x}, t, \xi) \Phi'(\xi) d\xi \, d\textbf{x} ds.
    \end{equation}
    
    \textbf{Source Term $I_4$:}
    By Hypothesis \ref{H4}, $\mathcal{S} \le M_{\mathcal{S}}$. Since $M = M_{\mathcal{S}} + M_K$ and $v \ge 0$, we have:
    \begin{equation}
        I_4 = \int_0^t \int_\Omega (\mathcal{S} - M) v d\textbf{x} ds \le \int_0^t \int_\Omega -M_K v d\textbf{x} ds.
    \end{equation}
    
    Combining these estimates and substituting $v = \int_{Y(t)}^u \Phi'(\xi) d\xi$, we analyze $I_4 - I_3$:
    \begin{equation}
        I_4 - I_3 \le \int_0^t \int_\Omega \int_{Y(t)}^u \left( \nabla_{\textbf{x}} \cdot \bar{K}(\textbf{x}, t, \xi) - M_K \right) \Phi'(\xi) d\xi \, d\textbf{x} ds.
    \end{equation}
    Given the assumption $|\nabla_{\textbf{x}} \cdot \bar{K}| \le M_K$, we know $\nabla_{\textbf{x}} \cdot \bar{K} - M_K \le 0$. Since $\Phi' \ge 0$, the entire integral is non-positive, yielding $I_4 - I_3 \le 0$.
    Since $I_1 \ge 0$ and $I_2 \ge 0$, the equality $I_1 + I_2 = I_4 - I_3 \le 0$ implies $I_1 = 0$, and consequently $u(\textbf{x},t) \le Y(t)$ almost everywhere.

    \textbf{Step 2: Lower Bound.}
    Let $Y(t)$ be the spatially constant subsolution defined by $\partial_t \theta(Y) = -M$, with $Y(0) = m_0$. 
    We define the non-positive test function $v = -[\Phi(u) - \Phi(Y(t))]_- = \min(0, \Phi(u) - \Phi(Y(t)))$. 
    Since $Y(t) \le m_0 \le u_d$ on the boundary, $v = 0$ on $\partial \Omega$.
    
    Testing the weak formulation with $v$ yields a similar equality $I_1 + I_2 + I_3 = I_4$.
    \begin{equation}
        I_1 = \int_\Omega \mathcal{E}(u(t), Y(t)) d\textbf{x} - \int_\Omega \mathcal{E}(u_0, Y(0)) d\textbf{x},
    \end{equation}
    where

    \begin{equation}
        \mathcal{E}(u, Y) = \int_Y^u \theta'(\xi) \left(-[\Phi(\xi) - \Phi(Y)]_-\right) d\xi.
    \end{equation}
    For $u < Y$, the integrand is negative and the integration direction is reversed, rendering $\mathcal{E} \ge 0$. Since $u_0 \ge m_0 = Y(0)$, the initial value is zero, so $I_1 \ge 0$. The diffusion term remains $I_2 \ge 0$.
    
    For the convection term, the primitive for $s < Y(t)$ yields 
    \begin{equation}
        \nabla_{\textbf{x}} \cdot \mathcal{K} = -\int_u^{Y(t)} \nabla_{\textbf{x}} \cdot \bar{K}(\mathbf{x},t,\xi)\Phi'(\xi) d\xi,
    \end{equation}
    leading to:
    \begin{equation}
        \begin{aligned}
            I_3 &= - \int_0^t \int_\Omega \left( -\int_u^{Y(t)} \nabla_{\textbf{x}} \cdot \bar{K}(\textbf{x}, t, \xi) \Phi'(\xi) d\xi \right) d\textbf{x} ds\\ &= \int_0^t \int_\Omega \int_u^{Y(t)} \nabla_{\textbf{x}} \cdot \bar{K}(\textbf{x}, t, \xi) \Phi'(\xi) d\xi \, d\textbf{x} ds.
        \end{aligned}
    \end{equation}
    
    For the source term, since $\mathcal{S} \ge -M_{\mathcal{S}}$, we have $\mathcal{S} + M \ge M_K$. Because $v \le 0$, we have:
    \begin{equation}
        \begin{aligned}
            I_4 = \int_0^t \int_\Omega (\mathcal{S} + M) v d\textbf{x} ds &\le \int_0^t \int_\Omega M_K v d\textbf{x} ds \\&= \int_0^t \int_\Omega \int_u^{Y(t)} -M_K \Phi'(\xi) d\xi \, d\textbf{x} ds.
        \end{aligned}
    \end{equation}
    
    Analyzing $I_4 - I_3$:
    \begin{equation}
        I_4 - I_3 \le \int_0^t \int_\Omega \int_u^{Y(t)} \left( -M_K - \nabla_{\textbf{x}} \cdot \bar{K}(\textbf{x}, t, \xi) \right) \Phi'(\xi) d\xi \, d\textbf{x} ds.
    \end{equation}
    Given $|\nabla_{\textbf{x}} \cdot \bar{K}| \le M_K$, we know $-M_K - \nabla_{\textbf{x}} \cdot \bar{K} \le 0$. Since $\Phi' \ge 0$, the integral is non-positive, yielding $I_4 - I_3 \le 0$. This leads to $I_1 = 0$, proving $u(\textbf{x},t) \ge Y(t)$ almost everywhere.
\end{proof}

To mathematically enforce the physical bounds of the saturation variable $0 \le \theta(u) \le 1$ when a source term is present, we must recognize a physical reality: a source cannot inject water into an already fully saturated medium, nor can a sink extract water from completely dry soil. If a constant positive source were applied indefinitely, the symmetric extension of $\theta$ utilized for the existence proofs would allow the mathematical solution to unphysically exceed $u^*$, where $u^*$ is defined such that $\theta(u^*)=1$. Therefore, to preserve the strict physical limits, the source term must be state-dependent and "shut off" when these bounds are reached. The following theorem establishes this result.

\begin{theorem}[Physical Maximum Principle]
Let $u$ be a weak solution to the Richards equation \eqref{uform} with initial data $u_0 \in L^\infty(\Omega)$ and boundary data $u_d \in L^\infty(\partial\Omega \times I)$. Define the global initial and boundary bounds:
\begin{equation}
    M_0 = \max\left( \operatorname*{ess\,sup}_{\textbf{x} \in \Omega} u_0(\textbf{x}), \ \operatorname*{ess\,sup}_{(\textbf{x},t) \in \partial \Omega \times I} u_d(\textbf{x},t) \right),
\end{equation}
\begin{equation}
    m_0 = \min\left( \operatorname*{ess\,inf}_{\textbf{x} \in \Omega} u_0(\textbf{x}), \ \operatorname*{ess\,inf}_{(\textbf{x},t) \in \partial \Omega \times I} u_d(\textbf{x},t) \right).
\end{equation}
Assume that the initial and boundary data reside entirely within the physical domain, such that $0 \le m_0 \le M_0 \le u^*$. 

Furthermore, assume there exist global constants $M_{\mathcal{S}}, M_K \ge 0$ such that the source term and the spatial divergence of the convective coefficient are bounded almost everywhere:
\begin{equation}
    |\mathcal{S}(\textbf{x}, t, u)| \le M_{\mathcal{S}} \quad \text{and} \quad |\nabla_{\textbf{x}} \cdot \bar{K}(\textbf{x}, t, u)| \le M_K.
\end{equation}
Finally, assume that these terms are state-dependent and respect the physical bounds:
\begin{enumerate}
    \item $\mathcal{S}(\textbf{x}, t, u) \le 0$ and $\nabla_{\textbf{x}} \cdot \bar{K}(\textbf{x}, t, u) \le 0$ for $u 
    \geq u^*$ (saturation limit),
    \item $\mathcal{S}(\textbf{x}, t, u) \ge 0$ and $\nabla_{\textbf{x}} \cdot \bar{K}(\textbf{x}, t, u) \ge 0$ for $u \leq 0$ (dry limit).
\end{enumerate}
Then, the solution remains strictly within the physical domain:
\begin{equation}
    0 \le u(\textbf{x},t) \le u^* \quad \text{for a.e. } (\textbf{x},t) \in \Omega \times I.
\end{equation}
\label{strictmaxprincipe}
\end{theorem}

\begin{proof}
    \textbf{Upper Bound:} 
    Let the supersolution be the constant physical maximum $Y = u^*$. We define the test function $v = [\Phi(u) - \Phi(u^*)]_+$. 
    Since $u_d \le M_0 \le u^*$, we have $\Phi(u_d) \le \Phi(u^*)$ on $\partial \Omega$, meaning $v = 0$ on the boundary and thus $v \in L^2(I; H_0^1(\Omega))$.
    
    Testing the weak formulation with $v$ and analyzing the terms $I_1 + I_2 + I_3 = I_4$ as in Theorem \ref{maxprincipe}:
    
    \textbf{Time Term $I_1$ and Diffusion Term $I_2$:}
    Following identical steps, $I_1 = \int_\Omega \mathcal{E}(u(t), u^*) d\textbf{x} - \int_\Omega \mathcal{E}(u_0, u^*) d\textbf{x}$. Since $u_0 \le M_0 \le u^*$, the initial term vanishes and $I_1 \ge 0$. Since $\nabla \Phi(u^*) = 0$, $I_2 \ge 0$.
    
    \textbf{Convection Term $I_3$:}
    We explicitly define the primitive $\mathcal{K}$ with respect to the variable $u$:
    \begin{equation}
        \mathcal{K}(\textbf{x}, t, s) = 
        \begin{cases}
            \int_{u^*}^s \bar{K}(\textbf{x}, t, \xi) \Phi'(\xi) d\xi & \text{if } s > u^*, \\
            0 & \text{if } s \le u^*.
        \end{cases}
    \end{equation}
    Applying Green's theorem, we isolate the spatial divergence:
    \begin{equation}
        I_3 = - \int_0^t \int_\Omega \nabla_{\textbf{x}} \cdot \mathcal{K} \, d\textbf{x} ds = - \int_0^t \int_\Omega \int_{u^*}^u \nabla_{\textbf{x}} \cdot \bar{K}(\textbf{x}, t, \xi) \Phi'(\xi) d\xi \, d\textbf{x} ds.
    \end{equation}
    
    \textbf{Source Term $I_4$:}
    The integral is $I_4 = \int_0^t \int_\Omega \mathcal{S}(\textbf{x}, t, u) v \, d\textbf{x} ds$. 
    
    Analyzing $I_4 - I_3$:
    \begin{equation}
        I_4 - I_3 = \int_0^t \int_\Omega \mathcal{S}(\textbf{x}, t, u) v \, d\textbf{x} ds + \int_0^t \int_\Omega \int_{u^*}^u \nabla_{\textbf{x}} \cdot \bar{K}(\textbf{x}, t, \xi) \Phi'(\xi) d\xi \, d\textbf{x} ds.
    \end{equation}
    The test function $v$ is strictly positive only where $u > u^*$. By our state-dependent hypothesis, for $\xi \ge u^*$, we have $\mathcal{S} \le 0$ and $\nabla_{\textbf{x}} \cdot \bar{K} \le 0$. Since $\Phi' \ge 0$ and $v \ge 0$, both integrands are non-positive, yielding $I_4 - I_3 \le 0$.
    
    Combining these yields $I_1 + I_2 \le I_4 - I_3 \le 0$, which requires $I_1 = 0$, implying $u(\textbf{x},t) \le u^*$ almost everywhere.
    
    \textbf{Lower Bound:} 
    The proof that $u \ge 0$ follows analogous similar arguments. We set the subsolution to $Y = 0$ and use the test function $v = -[\Phi(u) - \Phi(0)]_- = \min(0, \Phi(u))$. 
    Since $u_d \ge m_0 \ge 0$, $v = 0$ on $\partial \Omega$. For $u_0 \ge m_0 \ge 0$, the initial entropy vanishes, rendering $I_1 \ge 0$.
    
    The convective term evaluates to:
    \begin{equation}
        I_3 = \int_0^t \int_\Omega \int_u^0 \nabla_{\textbf{x}} \cdot \bar{K}(\textbf{x}, t, \xi) \Phi'(\xi) d\xi \, d\textbf{x} ds.
    \end{equation}
    Analyzing $I_4 - I_3$:
    \begin{equation}
        I_4 - I_3 = \int_0^t \int_\Omega \mathcal{S}(\textbf{x}, t, u) v \, d\textbf{x} ds - \int_0^t \int_\Omega \int_u^0 \nabla_{\textbf{x}} \cdot \bar{K}(\textbf{x}, t, \xi) \Phi'(\xi) d\xi \, d\textbf{x} ds.
    \end{equation}
    The state-dependent hypothesis ensures $\mathcal{S} \ge 0$ and $\nabla_{\textbf{x}} \cdot \bar{K} \ge 0$ for $\xi \le 0$. Since $v \le 0$, the first term is non-positive. Because the integral $\int_u^0 \dots \ge 0$, its subtraction is also non-positive. This results in $I_4 - I_3 \le 0$. Thus, $I_1 = 0$, forcing $u(\textbf{x},t) \ge 0$ almost everywhere.
\end{proof}
\begin{corollary}[Global Boundedness with Physical Limiters]
Assume the hypotheses of both Theorem \ref{maxprincipe} and Theorem \ref{strictmaxprincipe} hold simultaneously. Specifically, let the initial and boundary data reside within the physical domain such that $0 \le m_0 \le M_0 \le u^*$. Furthermore, assume the state-dependent source term $\mathcal{S}$ and convective spatial divergence $\nabla_{\textbf{x}} \cdot \bar{K}$ satisfy both the absolute global bounds ($M_{\mathcal{S}}$, $M_K$) and the physical state-dependent cutoffs at $u \le 0$ and $u \ge u^*$. 

Then, the weak solution $u$ satisfies the combined pointwise estimate:
\begin{equation}
    \max\left(0, \theta^{-1}(\theta(m_0) - Mt)\right) \le u(\textbf{x},t) \le \min\left(u^*, \theta^{-1}(\theta(M_0) + Mt)\right)
\end{equation}
for almost every $(\textbf{x},t) \in \Omega \times I$, where $M = M_{\mathcal{S}} + M_K$.
\end{corollary}

\begin{proof}
    By Theorem \ref{maxprincipe} (Generalized Maximum Principle), the temporal evolution of the solution is bounded by the dynamical expanding cone driven by the maximal effective source $M$:
    \begin{equation}
        \theta^{-1}(\theta(m_0) - Mt) \le u(\textbf{x},t) \le \theta^{-1}(\theta(M_0) + Mt) \quad \text{a.e. in } \Omega \times I.
    \end{equation}
    Simultaneously, by Theorem \ref{strictmaxprincipe} (Strict Physical Maximum Principle), the state-dependent cutoffs ensure that the solution is absolutely confined to the physical domain:
    \begin{equation}
        0 \le u(\textbf{x},t) \le u^* \quad \text{a.e. in } \Omega \times I.
    \end{equation}
    Because the weak solution $u$ must independently satisfy both sets of inequalities almost everywhere, it is bounded by their tightest constraints at any given time $t$. Taking the maximum of the lower bounds and the minimum of the upper bounds yields the desired intersection.
\end{proof}
\begin{remark}[Extension to Spatially Heterogeneous Media]
While the present analysis considers $\theta \equiv \theta(u)$ and $K \equiv K(u)$ to maintain consistency with the foundational existence results for equation \eqref{uform}, the logic of the Physical Maximum Principle extends naturally to spatially heterogeneous media where $\theta \equiv \theta(\textbf{x},u)$ and $K \equiv K(\textbf{x},u)$. 

Mathematically, applying the chain rule to such systems generates additional effective terms. Explicit spatial dependence introduces an effective convective flux tied to the spatial gradient of the parameters, such as $\nabla_{\textbf{x}} K(\textbf{x},u)|_u$. For the absolute physical bounds $0 \le u(\textbf{x},t) \le u^*$ to be preserved, this effective term must satisfy the same state-dependent limits as $\mathcal{S}$ and $\nabla_{\textbf{x}} \cdot \bar{K}$.
\end{remark}
\section{Applications}
\subsection{Richards' Equation}

In this section, we show that the existence result established in this paper applies to the Richards equation using the specific soil water retention models considered in \cite{benfanich2025}. For these models, we utilize the transformation defined by:
\begin{equation}
    \mathcal{U}(\theta) = \int_0^\theta (1-s^c)^{-b} \, ds,
\end{equation}
where $c\geq 1$ and $0\leq b <1$. The mapping $\mathcal{U}: [0,1] \to [0, u^*]$ is a bijection, where the constant $u^*$ is given by $u^* = \int_0^1 (1-s^c)^{-b} \, ds \geq 1$. The inverse function, the effective saturation $\theta = \mathcal{U}^{-1}$, lies in $C^1([0, u^*])$ and is strictly increasing. To satisfy the global existence hypotheses, we extend $\theta$ to $\mathbb{R}$ using the variable $\eta$ as follows:
\begin{equation}
    \theta(\eta) = 
    \begin{cases} 
        \eta & \text{if } \eta < 0, \\
        \mathcal{U}^{-1}(\eta) & \text{if } 0 \le \eta \le u^*, \\
        2 - \theta(2u^* - \eta) & \text{if } \eta > u^*.
    \end{cases}
\end{equation}
The extension for $\eta > u^*$ is constructed using the central symmetry of center $(u^*, 1)$ to preserve $C^1$ regularity. The behavior of this extension is illustrated in Figure \ref{fig:extensat}.

\begin{figure}[H]
    \centering
    \includegraphics[width=0.8\linewidth]{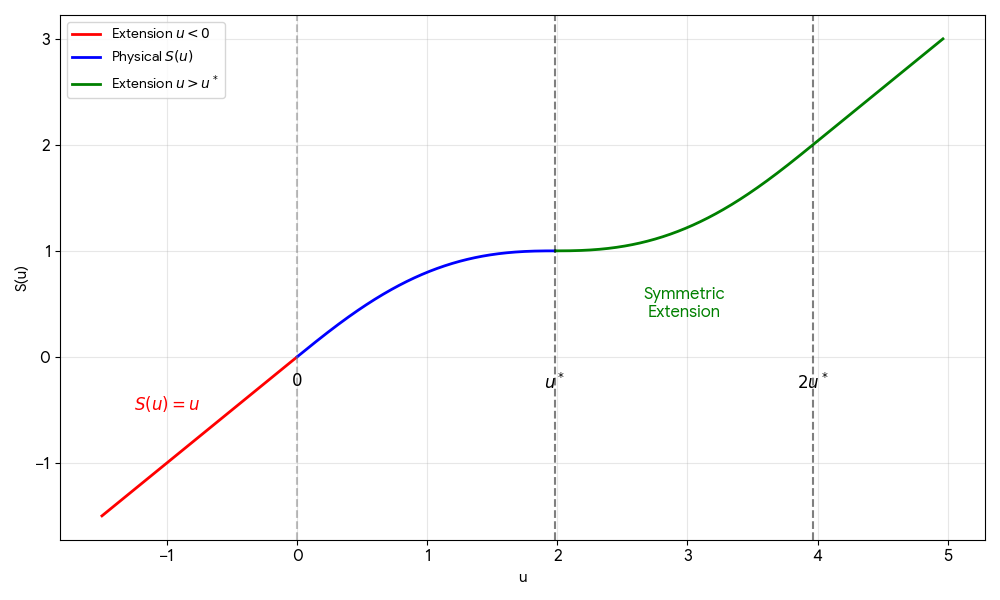}
    \caption{Extension of the saturation function $\theta(\eta)$ for parameters $c=\frac{5}{3}$ and $b = \frac{3}{5}$.}
    \label{fig:extensat}
\end{figure}

For a homogeneous medium, the constant $\phi = \theta_s - \theta_r$. To apply these results to the full physical model over the domain $\eta \in [0, u^*]$, the conductivity is defined as: 
\begin{equation}
    K(\eta) = \frac C \phi K_sK_r(\theta(\eta)) \theta(\eta)^{-a},
\end{equation}
where $K_s$ is the hydraulic conductivity, $C$ is a constant related to the hydraulic properties, $K_r$ is the relative permeability, and $a \ge 1$ is a parameter resulting from the change of variables. Note that while $\theta(\eta)^{-a}$ is singular at $\eta=0$, the product $K_r(\theta) \theta^{-a}$ stays bounded as $\theta \to 0$ under some conditions specified in \cite{benfanich2025}, rendering the singularity removable with $K(0)=0$.

We extend $K$ to $\mathbb{R}$ to verify Hypothesis \ref{H3} as follows:
\begin{enumerate}
    \item For $\eta > u^*$, we extend $K$ constantly by the saturated value:
    \begin{equation*}
        K(\eta) = C^* = \frac{C}{\phi} K_s K_r(\theta(u^*)) \theta(u^*)^{-a}.
    \end{equation*}
    \item For $\eta < 0$, we extend $K$ as an even function to preserve continuity at the origin:
    \begin{equation*}
        K(\eta) = K(-\eta).
    \end{equation*}
\end{enumerate}
Consequently, the global function $K: \mathbb{R} \to [0, K_{max}]$ is continuous and bounded. Figure \ref{fig:extK} illustrates the extended conductivity function $K$. We employ the standard van Genuchten-Mualem model \cite{VanGenuchten1980, mualem1976} for the soil hydraulic properties. The effective saturation $\theta$, the pressure head $\psi$, and relative permeability $K_r$ are related by:
\begin{equation}
    \psi(\theta) = -h_{cap} (\theta^{-\frac1m} -1)^{\frac1n}, \quad K_r(\theta) = {\theta}^{\frac12} \left[ 1 - (1 - \theta^{1/m})^m \right]^2,
\end{equation}
where $h_{cap}, n, m$ are empirical parameters with $m = 1 - 1/n$. 
For the mathematical analysis, these physical variables are transformed into the variable $u$ via the transformation described in \cite{benfanich2025}.

\begin{figure}[H]
    \centering
    \includegraphics[width=0.8\linewidth]{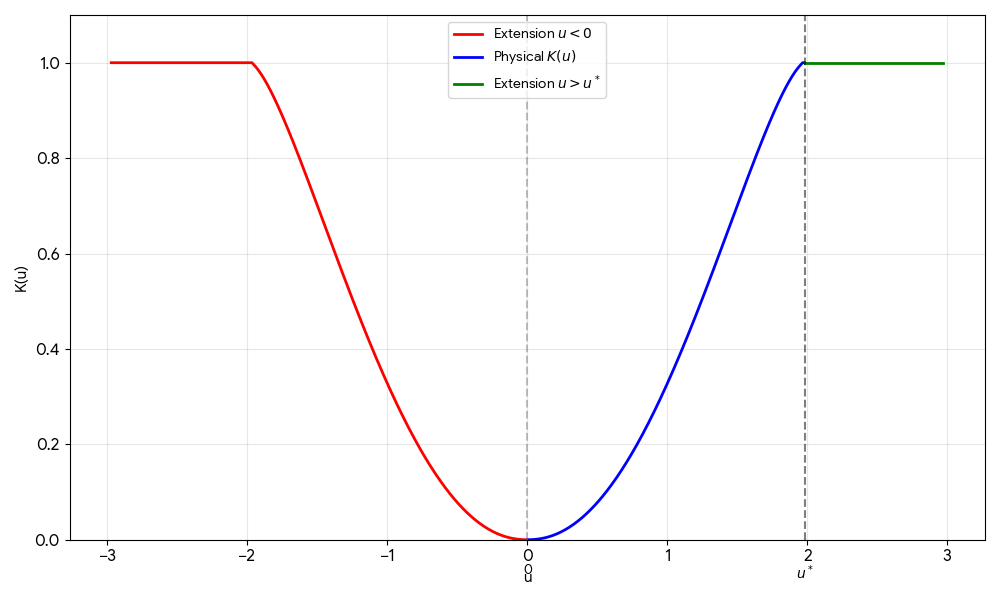}
    \caption{Extension of the diffusivity function $K(\eta)$ with parameters $K_s C = 1$, $m=0.6$, and $a=\frac 53$.}
    \label{fig:extK}
\end{figure}

Finally, we identify the convective term coefficient $\bar K$ and its factorization required by Hypothesis \ref{H3}. We define the auxiliary scaling function $\bar K_1$ as:
\begin{equation}
    \bar K_1(\eta) = \frac{1}{C}\theta(\eta)^a \boldsymbol{e}_z, \quad \text{for } \eta \in [0, u^*].
\end{equation}
We extend $\bar K_1$ to $\mathbb{R}$ similarly to $K$:
\begin{enumerate}
    \item For $\eta > u^*$, we set $\bar K_1(\eta) = \frac{1}{C}\boldsymbol{e}_z$.
    \item For $\eta < 0$, we set $\bar K_1(\eta) = \bar K_1(-\eta)$.
\end{enumerate}
The conductivity $\bar K$ (appearing in the gravity term) is then recovered via the decomposition:
\begin{equation}
    \bar K(\eta) = K(\eta) \bar K_1(\eta).
\end{equation}
This construction ensures that $\bar K$ is bounded and continuous, and that the factorization $\bar K = K \bar K_1$ holds globally. Figure~\ref{fig:extK1} displays the vertical ($z$) component of the vector function $\bar{K}$ derived from the van Genuchten-Mualem model.

\begin{figure}[H]
    \centering
    \includegraphics[width=0.8\linewidth]{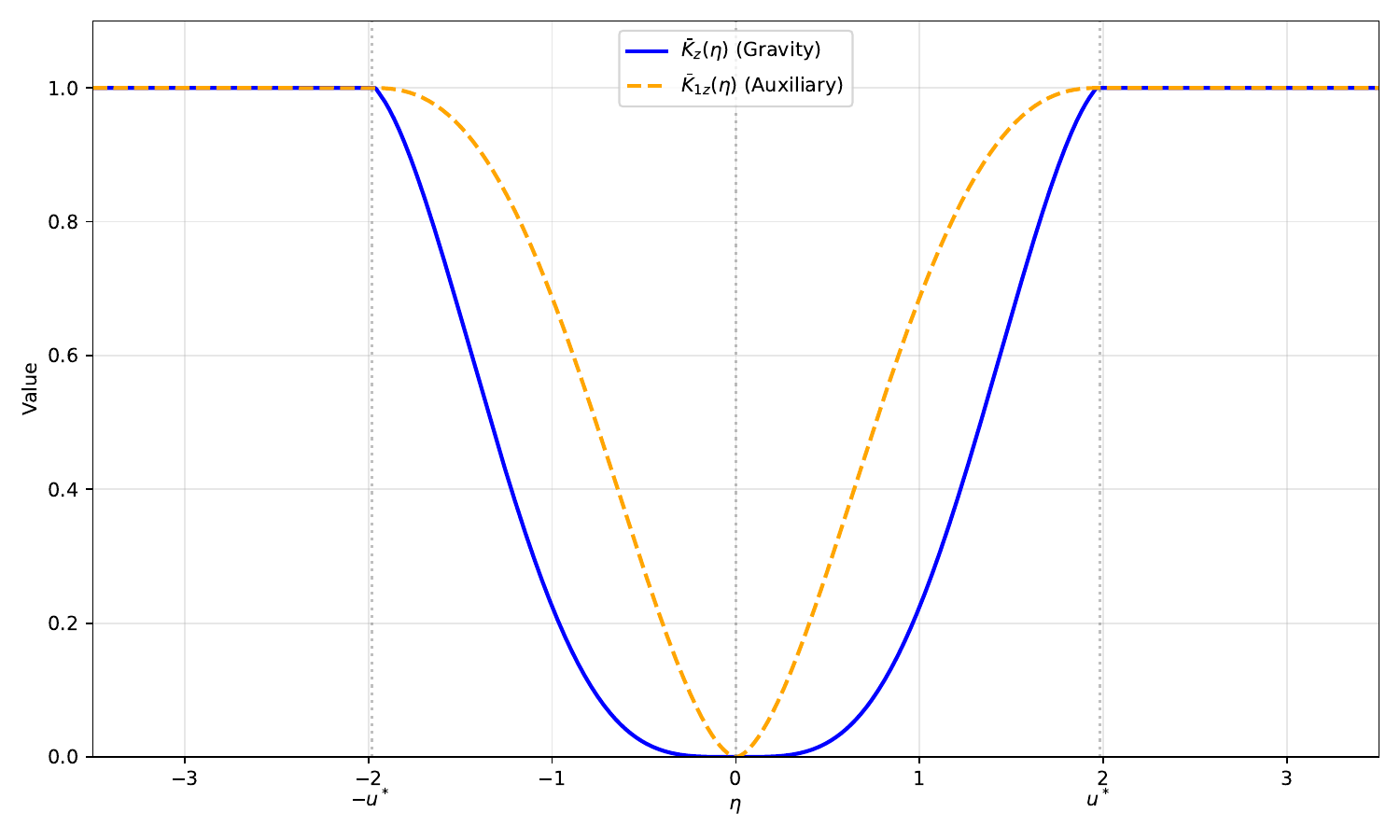}
    \caption{Comparison of the extended functions $\bar K_z$ and $\bar K_{1z}$.}
    \label{fig:extK1}
\end{figure}
We suppose that the source term $\mathcal{S}=0$.

\textbf{Verification of Hypotheses}

We now verify that the extended saturation function $\theta$ and the hydraulic conductivity $K$ constructed above satisfy the structural hypotheses \ref{H1}--\ref{H3} required for the existence theory.

\begin{proposition}[Verification of H1]
    Let $b \in [0, 1)$ and $c \ge 1$. The extended saturation function $\theta: \mathbb{R} \to \mathbb{R}$ satisfies Hypothesis \ref{H1}. Specifically:
    \begin{enumerate}
        \item $\theta(0)=0$ and $\theta$ is strictly increasing.
        \item $\theta \in C^1(\mathbb{R})$ and its derivative is bounded.
        \item The inverse function $\theta^{-1} = \mathcal{U}$ is uniformly Hölder continuous on the physical domain $[0, 1]$.
    \end{enumerate}
\end{proposition}

\begin{proof}
    $ $\newline
    \textbf{1. Origin:}
    For $\eta \in [0, u^*]$, $\theta(\eta)$ is the inverse of $\mathcal{U}(\theta) = \int_0^\theta (1-s^c)^{-b} ds$. We have $\mathcal U (0) = 0$, so we can conclude that $\theta(0) =0$.

    \textbf{2. Monotonicity and Regularity ($C^1$ and Bounded Derivative):}
    We compute the derivative $\theta'(\eta)$.
    For $\eta < 0$, $\theta'(\eta) = 1$.
    For $\eta \in [0, u^*]$, using the inverse function theorem:
    \begin{equation}
        \theta'(\eta) = \frac{1}{\mathcal{U}'(\theta)} = (1 - \theta(\eta)^c)^b.
    \end{equation}
    At $\eta=0$ (where $\theta=0$), $\theta'(0^+) = 1$, matching the left derivative.
    At $\eta=u^*$ (where $\theta=1$), $\theta'(u^{*-}) = 0$.
    For $\eta > u^*$, differentiating the symmetry relation yields $\theta'(\eta) = \theta'(2u^* - \eta)$. Thus $\theta'(u^{*+}) = 0$, ensuring $C^1$ continuity.
    Since $\theta(\eta) \in [0, 1]$, we have $0 \le \theta'(\eta) \le 1$, so the derivative is bounded. Moreover, the derivative in all of these cases is non-negative and vanishes at only one point, so we can conclude that the function is strictly increasing.

    \textbf{3. H\"older Continuity of the Inverse:}
    We aim to show that $\mathcal{U}$ is H\"older continuous on the interval $[0, 2]$. Let $\theta_1, \theta_2 \in [0, 2]$. Without loss of generality, assume $\theta_1 < \theta_2$. We distinguish three cases based on the location of the points relative to the symmetry axis $\theta=1$.

    \textit{Case 1: $0 \le \theta_1 < \theta_2 \le 1$.}
    \begin{equation}
        \begin{aligned}
            |\mathcal{U}(\theta_1)-\mathcal{U}(\theta_2)| &= \int_{\theta_1}^{\theta_2} (1-s^c)^{-b}ds \\
            & \leq \int_{\theta_1}^{\theta_2} (1-s)^{-b}ds \quad (\text{since } s^c \le s \text{ for } c \ge 1, s \in [0,1]) \\
            & = \frac{1}{1-b} \left((1-\theta_1)^{1-b} - (1-\theta_2)^{1-b}\right).
        \end{aligned}
    \end{equation}
    Using the inequality $\left||x|^\gamma - |y|^\gamma\right| \le |x-y|^\gamma$ for $0 < \gamma \le 1$ (with $\gamma = 1-b$), we obtain:
    \begin{equation}
        |\mathcal{U}(\theta_1)-\mathcal{U}(\theta_2)| \le \frac{1}{1-b} |(1-\theta_1) - (1-\theta_2)|^{1-b} = \frac{1}{1-b} |\theta_2 - \theta_1|^{1-b}.
    \end{equation}

    \textit{Case 2: $1 \le \theta_1 < \theta_2 \le 2$.}
    By the symmetry of the construction, $\mathcal{U}(\theta)$ on $[1, 2]$ is a reflection of the behavior on $[0, 1]$. Specifically, let $\tilde{\theta}_1 = 2-\theta_2$ and $\tilde{\theta}_2 = 2-\theta_1$. Then $0 \le \tilde{\theta}_1 < \tilde{\theta}_2 \le 1$, and $|\mathcal{U}(\theta_1) - \mathcal{U}(\theta_2)| = |\mathcal{U}(\tilde{\theta}_1) - \mathcal{U}(\tilde{\theta}_2)|$. Applying Case 1:
    \begin{equation}
        |\mathcal{U}(\theta_1)-\mathcal{U}(\theta_2)| \le \frac{1}{1-b} |\tilde{\theta}_2 - \tilde{\theta}_1|^{1-b} = \frac{1}{1-b} |\theta_2 - \theta_1|^{1-b}.
    \end{equation}

    \textit{Case 3: $\theta_1 < 1 < \theta_2$.}
    We use the triangle inequality by inserting the point $1$:
    \begin{equation}
        |\mathcal{U}(\theta_1)-\mathcal{U}(\theta_2)| \le |\mathcal{U}(\theta_1)-\mathcal{U}(1)| + |\mathcal{U}(1)-\mathcal{U}(\theta_2)|.
    \end{equation}
    Applying the results from Case 1 and Case 2 to each term:
    \begin{equation}
        |\mathcal{U}(\theta_1)-\mathcal{U}(\theta_2)| \le \frac{1}{1-b} |1 - \theta_1|^{1-b} + \frac{1}{1-b} |\theta_2 - 1|^{1-b}.
    \end{equation}
    Since $0 < 1-b \le 1$, the function $x \mapsto x^{1-b}$ is subadditive (concave). Therefore, $x^{1-b} + y^{1-b} \le 2^{b} (x+y)^{1-b} \le 2 (x+y)^{1-b}$ for $x,y \ge 0$. Thus:
    \begin{equation}
        |\mathcal{U}(\theta_1)-\mathcal{U}(\theta_2)| \le \frac{2}{1-b} ( (1-\theta_1) + (\theta_2-1) )^{1-b} = \frac{2}{1-b} |\theta_2 - \theta_1|^{1-b}.
    \end{equation}

    Combining all cases, we conclude that $\mathcal{U}$ is uniformly H\"older continuous on $[0, 2]$ with exponent $\delta = 1-b$ and constant $C_H = \frac{2}{1-b}$.
    
    \textbf{Global Verification of \ref{H1}:}
    To verify the global property in Hypothesis \ref{H1}, let $\zeta, \eta \in \mathbb{R}$. Recall that $\mathcal{U} = \theta^{-1}$. We let $\theta_1 = \theta(\zeta)$ and $\theta_2 = \theta(\eta)$. We check the inequality $|\zeta - \eta| = |\mathcal{U}(\theta_1) - \mathcal{U}(\theta_2)|$.
    
    If $\theta_1, \theta_2 \in [0, 2]$, we have shown $|\mathcal{U}(\theta_1) - \mathcal{U}(\theta_2)| \le C_H |\theta_1 - \theta_2|^\delta$.
    If $\theta_1, \theta_2 \notin [0, 2]$, then $\mathcal{U}$ acts as the identity (shifted), which is $1$-Lipschitz: $|\mathcal{U}(\theta_1) - \mathcal{U}(\theta_2)| \le |\theta_1 - \theta_2|$.
    
    For the mixed case, assume $\theta_1 \in [0, 2]$ and $\theta_2 > 2$. Using the triangle inequality through the boundary point $2$:
    \begin{equation}
        \begin{aligned}
             |\mathcal{U}(\theta_1) - \mathcal{U}(\theta_2)| &\le |\mathcal{U}(\theta_1) - \mathcal{U}(2)| + |\mathcal{U}(2) - \mathcal{U}(\theta_2)| \\
             &\le C_H |\theta_1 - 2|^\delta + |2 - \theta_2|.
        \end{aligned}
    \end{equation}
    Since $2$ lies between $\theta_1$ and $\theta_2$, we have $|\theta_1 - 2| \le |\theta_1 - \theta_2|$ and $|2 - \theta_2| \le |\theta_1 - \theta_2|$. Thus:
    \begin{equation}
        |\mathcal{U}(\theta_1) - \mathcal{U}(\theta_2)| \le C_H |\theta_1 - \theta_2|^\delta + |\theta_1 - \theta_2|.
    \end{equation}
    Setting $H_\theta = \max(1, C_H)$, we satisfy the condition for all $\zeta, \eta \in \mathbb{R}$:
    \begin{equation}
        |\zeta - \eta| \le H_\theta \left( |\theta(\zeta) - \theta(\eta)|^\delta + |\theta(\zeta) - \theta(\eta)| \right).
    \end{equation}
    This confirms that the extended function satisfies Hypothesis \ref{H1}.
    
    \textbf{4. Sobolev Regularity of the Inverse:}
    Finally, we show that $\mathcal{U} \in W^{1,1}_{loc}(\mathbb{R})$. Since $\mathcal{U}$ is continuous on $\mathbb{R}$, it suffices to verify that its classical derivative $\mathcal{U}'$ exists almost everywhere and belongs to $L^1_{loc}(\mathbb{R})$.
    
    The derivative is given piecewise by:
    \begin{equation}
        \mathcal{U}'(\theta) = \begin{cases}
            1 & \text{if } \theta < 0, \\
            (1-\theta^c)^{-b} & \text{if } 0 < \theta < 1, \\
            (1-(2-\theta)^c)^{-b} & \text{if } 1 < \theta < 2, \\
            1 & \text{if } \theta > 2.
        \end{cases}
    \end{equation}
    The only potential singularities occur at the points $\theta=1$ (and by symmetry in the extension). We analyze the behavior near $\theta=1^-$. As $\theta \to 1$, we use the Taylor expansion $1-\theta^c \approx c(1-\theta)$. Thus:
    \begin{equation}
        \mathcal{U}'(\theta) \approx \frac{1}{c^b} (1-\theta)^{-b}.
    \end{equation}
    Since the parameter satisfies $b < 1$, the function $(1-\theta)^{-b}$ is integrable near $\theta=1$. Specifically:
    \begin{equation}
        \int_{1-\epsilon}^1 (1-\theta)^{-b} d\theta = \left[ \frac{-(1-\theta)^{1-b}}{1-b} \right]_{1-\epsilon}^1 = \frac{\epsilon^{1-b}}{1-b} < \infty.
    \end{equation}
    Due to the symmetry of the extension, the integral is also finite approaching $\theta=1$ from the right. Outside the interval $[0, 2]$, the derivative is constant ($1$), which is locally integrable.
    
    Therefore, $\mathcal{U}' \in L^1_{loc}(\mathbb{R})$. Thus we conclude that $\mathcal{U} \in W^{1,1}_{loc}(\mathbb{R})$.
\end{proof}

\begin{proposition}[Verification of \ref{H2}]
    The extended saturation function satisfies the growth condition $\theta(\eta)\eta \ge \alpha \eta^2$ for some $\alpha > 0$.
\end{proposition}

\begin{proof}
    We observe that $\theta(\eta)$ is concave on $[0, u^*]$. On the interval $[u^*, 2u^*]$, we have $\theta(\eta) \ge 1$. We define the auxiliary function $g$:
    \begin{equation}
    g(\eta)=\begin{cases}
            \theta(\eta) & 0 \le \eta \le u^*, \\
            1 & u^* \le \eta \le 2u^*.
        \end{cases}
    \end{equation}
    Since $g$ is concave and $g(0)=0$ with $g(\eta)>0$, the secant slope $g(\eta)/\eta$ is bounded from below by the slope at the endpoint $2u^*$:
    \begin{equation}
        g(\eta) \ge \frac{g(2u^*)}{2u^*} \eta = \frac{1}{2u^*} \eta =: \alpha \eta.
    \end{equation}
    Since $\theta(\eta) \ge g(\eta)$, we have $\theta(\eta)\eta \ge \alpha \eta^2$.
    For $\eta \notin [0, 2u^*]$, the function behaves linearly ($\theta(\eta) = \eta$), which satisfies the inequality with $\alpha  = \frac{1}{2u^*}\leq 1$. Thus, the condition holds globally.
\end{proof}

\begin{remark}
    Consequently, Theorem \ref{theorem1} guarantees the existence of a weak solution to the Richards equation. Furthermore, we can apply the Physical Maximum Principle (Theorem \ref{strictmaxprincipe}) to this application. Because we consider a homogeneous medium with no source term, we have $\mathcal{S} = 0$ and $\nabla_{\mathbf{x}} \cdot \bar{K} = 0$. These trivially satisfy both the global bounds and the state-dependent cutoff hypotheses required by the theorem. 

Therefore, provided that the initial and boundary data reside within the physical domain ($0 \le u_0, u_d \le u^*$ almost everywhere), the theorem ensures that the solution remains strictly within the bounds $0 \le u(\mathbf{x},t) \le u^*$ almost everywhere. Equivalently, the effective saturation satisfies the physical constraints $0 \le \theta(u) \le 1$ almost everywhere. This shows that the solution remains strictly within the physical domain and is therefore completely independent of the specific extensions constructed for the unphysical regimes ($u < 0$ and $u > u^*$).
\end{remark}
\section{Convergence of the \texorpdfstring{$L$}{L}-scheme Linearization}
\begin{theorem}[Convergence of the L-scheme]
Let $u_n$ be the solution to the semi-discrete problem \eqref{semidesceq}. Let $\{u_n^i\}_{i\geq 1}$ be the sequence generated by the linear iteration scheme:
\begin{equation}
    L(u_n^{i}- u_n^{i-1},v) + (\theta(u_n^{i-1}),v)+\tau (u_n^i,v)_\omega =\langle f,v \rangle, \quad \forall v \in V,
    \label{Lscheme_thm}
\end{equation}
given an initial guess $u_n^0 \in L^2(\Omega)$.
Assume that the stabilization parameter satisfies $L > L_\theta/2$, where $L_\theta$ is the Lipschitz constant of $\theta$. Then, as $i \to \infty$,
\begin{enumerate}
    \item The sequence converges linearly: $u_n^i \to u_n$ in $L^2(\Omega)$.
    \item The gradients converge: $\sqrt{K(u_{n-1})} \nabla u_n^i \to \sqrt{K(u_{n-1})} \nabla u_n$ in $L^2(\Omega)$.
\end{enumerate}
\end{theorem}

\begin{proof}
    Let $e^i = u_n - u_n^i$ denote the error at iteration $i$. Subtracting the linearized equation \eqref{Lscheme_thm} from the exact semi-discrete equation \eqref{semidesceq}, we obtain the error equation:
    \begin{equation} 
        L(e^i -e^{i-1},v) + (\theta(u_n) - \theta(u^{i-1}_n),v) + \tau(e^i,v)_\omega = 0.
    \end{equation}
    We test this equation with $v=e^i$. Applying the algebraic identity $$(x-y)x = \frac{1}{2}(x^2 -y^2 + (x-y)^2)$$ to the first term, we find:
    \begin{equation}
        \begin{aligned}
            \frac{L}{2}(\|e^i\|^2 - \|e^{i-1}\|^2 + \| e^i -e^{i-1}\|^2) &+ (\theta(u_n) -\theta(u_n^{i-1}), e^{i-1}) + \tau|e^i|_V^2 \\
            &= -(\theta(u_n) -\theta(u_n^{i-1}), e^i - e^{i-1}).
        \end{aligned}
    \end{equation}
    Using the monotonicity and Lipschitz continuity of $\theta$ (denoted by $L_\theta$), combined with Cauchy's and Young's inequalities, we estimate the right-hand side:
    \begin{equation}
        \begin{aligned}
            \frac{L}{2}(\|e^i\|^2 - \|e^{i-1}\|^2 &+ \| e^i -e^{i-1}\|^2) + \frac{1}{L_\theta}\|\theta(u_n) -\theta(u_n^{i-1})\|^2 + \tau |e^i|_V^2 \\
            &\leq \frac{1}{2L}\|\theta(u_n) -\theta(u_n^{i-1})\|^2  + \frac{L}{2}\|e^i - e^{i-1}\|^2.
        \end{aligned}
    \end{equation}
    Simplifying the terms yields the fundamental error inequality:
    \begin{equation}
        \frac{L}{2}(\|e^i\|^2 - \|e^{i-1}\|^2 ) +\left(\frac{1}{L_\theta}-\frac{1}{2L}\right)\|\theta(u_n) -\theta(u_n^{i-1})\|^2+ \tau |e^i|_V^2 \leq 0.
        \label{Erroranalysis_proof}
    \end{equation}
    Summing this inequality for $1 \leq i \leq k$ and assuming $L > L_\theta/2$, we obtain the bound:
    \begin{equation}
        \frac{L}{2}\|e^k\|^2  +\left(\frac{1}{L_\theta}-\frac{1}{2L}\right)\sum_{i=1}^k\|\theta(u_n) -\theta(u_n^{i-1})\|^2+ \tau \sum_{i=1}^k|e^i|_V^2 \leq \frac{L}{2}\|e^0\|^2.
    \end{equation}
    Since the right-hand side is finite, the series on the left-hand side converge. This implies the convergence of the physical variables:
    \begin{equation}
        \begin{aligned}
        \theta(u_n^i) &\to \theta(u_n) \\
         \sqrt{K(u_{n-1})} \nabla u_n^i &\to \sqrt{K(u_{n-1})} \nabla u_n
    \end{aligned} \quad \text{as } i \to \infty, \text{ in } L^2(\Omega).
    \end{equation}
    
    To prove the strong convergence of $u_n^i$, we utilize the properties of the inverse function $\theta^{-1}$. 
    By Hypothesis \ref{H1}, we have:
    \begin{equation}
        \int_{\Omega} |u_n - u^i_n|^2 d\textbf{x} \leq 2H_\theta^2 \qty(\int_\Omega |\theta(u_n) -\theta(u_n^i)|^{2} d\textbf{x}+\int_\Omega |\theta(u_n) -\theta(u_n^i)|^{2\delta} d\textbf{x}).
    \end{equation}
    Applying H\"older's inequality with exponents $p=1/\delta$ and $q = 1/(1-\delta)$ yields:
    \begin{equation}
        \int_{\Omega} |u_n - u^i_n|^2 d\textbf{x} \leq  2H_\theta^2 \qty(\|\theta(u_n) -\theta(u_n^i)\|^{2}+|\Omega|^{1-\delta} \|\theta(u_n) -\theta(u_n^i)\|^{2\delta}).
    \end{equation}
    thus we have 
    \begin{equation}
        \|u_n-u_n^i\| \leq H_\theta \sqrt2\sqrt{ |\Omega|^{1-\delta} + \|\theta(u_n) -\theta(u_n^i)\|^{2(1-\delta)}} \|\theta(u_n) -\theta(u_n^i)\|^\delta
    \end{equation}
    Consequently, $u_n^i \to u_n$ strongly in $L^2(\Omega)$ as $i \to \infty$.
    
    Finally, we estimate the rate of convergence. From \eqref{Erroranalysis_proof} and \ref{H1}, for $i \geq i_0$:
    \begin{equation}
        \frac{L}{2}\|e^i\|^2 + \frac{1}{2^{\frac1\delta}H^{\frac 2\delta}_\theta \qty(|\Omega|^{1-\delta}+ \|\theta(u_n) -\theta(u_n^{i-1})\|^{2(1-\delta)})^{\frac 1\delta}}\left(\frac{1}{L_\theta}-\frac{1}{2L} \right)\|e^{i-1}\|^{\frac{2}{\delta}} \leq \frac{L}{2}\|e^{i-1}\|^2.
    \end{equation}
    Rearranging terms provides the recurrence relation:
    \begin{equation}
        \|e^i\|^2  \leq \left(1-\frac{2}{L2^{\frac1\delta}H^{\frac 2\delta}_\theta \qty(|\Omega|^{1-\delta}+ \|\theta(u_n) -\theta(u_n^{i-1})\|^{2(1-\delta)})^{\frac 1\delta}}\left(\frac{1}{L_\theta}-\frac{1}{2L} \right)\|e^{i-1}\|^{\frac{2(1-\delta)}{\delta}}\right)\|e^{i-1}\|^2.
    \end{equation}
    Thus, as $i \to \infty$, we have
    \begin{equation}
        \limsup_i\frac{\|e^i\|}{\|e^{i-1}\|} \leq 1,
    \end{equation}
    then, the convergence is at least linear, but with a ratio that gets close to $1$, which makes convergence linear but slow.
\end{proof}
\section{Convergence of the semi-implicit method}
\subsection{Convergence of the Regularized Discrete Solution}
In this subsection, we establish that the solution of the regularized semi-discrete problem converges to the solution of the degenerate semi-discrete problem as the regularization parameter $\epsilon \to 0$.

Let $u^\tau$ be the solution of the semi-implicit discretized equation \eqref{semiimplicitmeth}, defined piecewise as $u^{\tau}(t) = u_n$ for $t \in (t_{n-1}, t_n]$. Similarly, let $u^{\tau,\epsilon}$ be the solution to the regularized problem \eqref{regdiffconv_prop}.

\begin{proposition}
    Let $N \in \mathbb{N}$ and $\tau = T/N$. Let $\{u_n\}_{n=1}^N \subset L^2(\Omega)$, such that for every $n=1,\cdots, N$, we have $u_n \in V_n$ be the solution sequence of the degenerate scheme:
    \begin{equation}
        (\theta(u_n)- \theta(u_{n-1}),v) + \tau (K(u_{n-1})\nabla u_n,\nabla v) + \tau(\bar K(u_{n-1}), \nabla v) = \tau (\mathcal S(u_{n-1}),v),\; \forall v \in V_n,
        \label{semiimplicitmeth}
    \end{equation}
    where $V_n$ is the weighted Sobolev space with weight $\omega_n = \sqrt{K(u_{n-1})}$.
    
    Let $\epsilon > 0$ and $\{u_n^\epsilon\}_{n=1}^N \subset \mathcal{H}$ be the solution sequence of the regularized scheme:
    \begin{equation}
        \begin{aligned}
             (\theta(u^\epsilon_n)- \theta(u^\epsilon_{n-1}),v) &+ \tau ((K(u^\epsilon_{n-1})+\epsilon)\nabla u_n^\epsilon,\nabla v) \\
             &+ \tau(\bar K(u^\epsilon_{n-1}), \nabla v) = \tau (\mathcal S(u^\epsilon_{n-1}),v),\; \forall v \in \mathcal{H}.
        \end{aligned}
        \label{regdiffconv_prop}
    \end{equation}
    Then, for every $n = 1, \dots, N$, let $u^{\tau} = u_n$ for $t_{n-1}< t \leq t_{n}$ then:
    \begin{equation}
        \|u^{\tau,\epsilon} - u^\tau\|_{L^2(\Omega\times I)} \to 0 \quad \text{as } \epsilon \to 0,
    \end{equation}
    uniformly in terms of $0<\tau \leq T$
    \label{convoftheregsoltodegsol_disc}
\end{proposition}

\begin{proof}
    We proceed by induction on $n$.
    \textbf{Base case ($n=1$):}
    Both schemes share the same initial data, $u_0^\epsilon = u_0$. Consequently, the coefficients for the first step are identical: $K(u_0^\epsilon) = K(u_0)$.

    \textbf{Step 1: A Priori Estimates}
    Testing the regularized equation \eqref{regdiffconv_prop} with $v=u_1^\epsilon$ for $n=1$, we obtain:
    \begin{equation}
        \begin{aligned}
             (\theta(u^\epsilon_1),u_1^\epsilon) &+ \tau |u_1^\epsilon|_{V_1}^2 + \tau \epsilon \|\nabla u^\epsilon_1\|^2\\
             & = (\theta(u_{0}),u_1^\epsilon)+ \tau (\mathcal S(u_{0}),u_1^\epsilon)- \tau(\omega_1\bar K_1(u_{0}), \omega_1\nabla u_1^\epsilon).
        \end{aligned}
        \label{semiimplicit1testu1}
    \end{equation}
    Using Hypotheses \ref{H2}-\ref{H3}, along with the Cauchy-Schwarz and Young's inequalities, we estimate the terms as follows:
    \begin{equation}
        \begin{aligned}
             \alpha \|u_1^\epsilon\|^2 &+ \tau |u_1^\epsilon|_{V_1}^2 + \tau \epsilon \|\nabla u^\epsilon_1\|^2\\
             & \leq (\theta(u_{0}),u_1^\epsilon)+ \tau (\mathcal S(u_{0}),u_1^\epsilon)- \tau(\omega_1\bar K_1(u_{0}), \omega_1\nabla u_1^\epsilon)\\
             &\leq \qty(L_\theta \|u_0\| + \tau |\Omega|^{\frac 12} M_S) \|u_1^\epsilon\| + \tau |\Omega|^{\frac{1}{2}}M |u_1^\epsilon|_{V_1}\\
             & \leq \frac{1}{2\alpha }\qty(L_\theta \|u_0\| + \tau \Omega^{\frac 12} M_S)^2 + \frac{\alpha }{2}\|u_1^\epsilon\|^2+ \frac \tau2 |\Omega|M^2 + \frac \tau 2|u_1^\epsilon|_{V_1}^2.
        \end{aligned}
    \end{equation}
    Rearranging the terms, we find:
    \begin{equation}
        \frac{\alpha}{2} \|u_1^\epsilon\|^2 + \frac{\tau}{2} |u_1^\epsilon|_{V_1}^2 + \tau \epsilon \|\nabla u^\epsilon_1\|^2 \leq \frac{1}{2\alpha }\qty(L_\theta \|u_0\| + \tau \Omega^{\frac 12} M_S)^2+ \frac{\tau}{2} |\Omega|M^2.
    \end{equation}
    Thus, there exists a constant $C>0$, independent of $\epsilon$, such that:
    \begin{equation}
        \|u_1^\epsilon\| + \sqrt{\tau}|u_1^\epsilon|_{V_1}+ \sqrt{ \tau\epsilon} \|\nabla u_1^\epsilon\| \leq C.
    \end{equation}

    \textbf{Step 2: Weak Convergence}
    Due to these uniform bounds, there exists $u'_1 \in V_1$ such that (up to a subsequence):
    \begin{equation}
        \begin{aligned}
            &u^\epsilon_1 \rightharpoonup u'_1 \text{ weakly in } L^2(\Omega),\\
            & \sqrt{\tau}\omega_1 \nabla u^\epsilon_1 \rightharpoonup \sqrt\tau\omega_1 \nabla u'_1 \text{ weakly in } L^2(\Omega),\\
            & \epsilon\sqrt\tau \nabla u^\epsilon_1 \to 0 \text{ strongly in } L^2(\Omega).
        \end{aligned}
    \end{equation}
    Additionally, since $\theta$ is Lipschitz continuous, $\{\theta(u^\epsilon_1)\}$ is bounded in $L^2(\Omega)$. Therefore, there exists $w_1 \in L^2(\Omega)$ such that:
    \begin{equation}
        \theta(u^\epsilon_1) \rightharpoonup w_1 \text{ weakly in } L^2(\Omega).
    \end{equation}
    Proceeding as in the existence proof (Theorem \ref{eistenceforsemidis}), we identify the limit and show that $u'_1$ is a solution to the equation \eqref{semiimplicitmeth} for $n=1$. Since the solution to the limit problem is unique (by Theorem \ref{eistenceforsemidis}), we conclude that the entire sequence converges weakly: $u^\epsilon_1 \rightharpoonup u_1$ in $L^2(\Omega)$.

    \textbf{Step 3: Strong Convergence}
    It remains to show that $u^\epsilon_1 \to u_1$ strongly in $L^2(\Omega)$. By Lemma \ref{Lemmabres}, we have:
    \begin{equation}
        \lim_{\epsilon \to 0} (\theta(u^\epsilon_1),u^\epsilon_1) = (\theta(u_1),u_1).
    \end{equation}
    By the monotonicity of $\theta$, we consider:
    \begin{equation}
        0\leq \int_\Omega (\theta(u^\epsilon_1)-\theta(u_1))(u^\epsilon_1-u_1)\; d\textbf{x}.
    \end{equation}
    Expanding this product and using the weak convergence results, we obtain:
    \begin{equation}
        \int_\Omega (\theta(u^\epsilon_1)-\theta(u_1))(u^\epsilon_1-u_1)\; d\textbf{x} = (\theta(u^\epsilon_1),u^\epsilon_1) - (\theta(u_1),u^\epsilon_1)-(\theta(u^\epsilon_1),u_1) + (\theta(u_1),u_1) \underset{\epsilon \to 0 }{\to} 0.
    \end{equation}
    Finally, using the fact that $\theta$ is Lipschitz, we have:
    \begin{equation}
        \frac{1}{L_\theta} \|\theta(u^\epsilon_1)-\theta(u_1)\|^2 \leq \int_\Omega (\theta(u^\epsilon_1)-\theta(u_1))(u^\epsilon_1-u_1)\; d\textbf{x} \to 0.
    \end{equation}
    Combining this with Hypothesis \ref{H1}, we conclude that $u^\epsilon_1 \to u_1$ strongly in $L^2(\Omega)$.
    
    \textbf{Induction step:} Assume that the result holds for the previous step, i.e., $\|u^\epsilon_{n-1} - u_{n-1}\|_{L^2(\Omega)} \to 0$. We aim to show that $\|u^\epsilon_{n} - u_{n}\|_{L^2(\Omega)} \to 0$.

Proceeding as in the base case, we derive the following uniform estimates. There exists a constant $C>0$ such that for all $\epsilon>0$:
\begin{equation}
    \|u^\epsilon_n\| + \sqrt\tau\|\sqrt{K(u_{n-1}^\epsilon)} \nabla u^\epsilon_n\| + \sqrt {\tau\epsilon} \|\nabla u^\epsilon_n\| \leq C.
\end{equation}
From these estimates, we deduce the existence of $u'_n \in L^2(\Omega)$ and $\bar Z \in (L^2(\Omega))^d$ such that (up to a subsequence):
\begin{equation}
    \begin{aligned}
        & u^\epsilon_n \rightharpoonup u'_n \text{ weakly in } L^2(\Omega),\\
        & \theta(u^\epsilon_n) \rightharpoonup w_n \text{ weakly in } L^2(\Omega),\\
        & \sqrt\tau\sqrt{K(u_{n-1}^\epsilon)} \nabla u^\epsilon_n \rightharpoonup \bar Z \text{ weakly in } (L^2(\Omega))^d,\\
        & \epsilon\sqrt\tau \nabla u^\epsilon_n \to 0 \text{ strongly in } L^2(\Omega).
    \end{aligned}
\end{equation}
Using the induction hypothesis, along with the continuity and boundedness of $K$, we have (up to a subsequence):
\begin{equation}
    K(u^\epsilon_{n-1}) \to K(u_{n-1}) \text{ strongly in } L^2(\Omega) \text{ and a.e.}
\end{equation}

We now need to identify the limit flux as $\bar Z = \sqrt \tau \sqrt{K(u_{n-1})}\nabla u'_n$ on the set where the weight is not zero.
Let $m \in \mathbb N$. Consider the sets $\Omega^m = \{\textbf{x} \in \Omega : K(u_{n-1}(\textbf{x})) \geq \frac 1m \}$, $\Omega^+ = \bigcup_{m\geq 1} \Omega^m$, and $\Omega_0 = \{\textbf{x} \in \Omega : K(u_{n-1}(\textbf{x})) = 0\} $. Note that $\Omega$ is the disjoint union of $\Omega_0$ and $\Omega^+$.

On the degenerate set $\Omega_0$, since $K(u_{n-1}) = 0$, the strong convergence implies $K(u^\epsilon_{n-1}) \to 0$ almost everywhere, which yields $\sqrt{K(u^\epsilon_{n-1})} \to 0$ in $L^2(\Omega_0)$.
For any test function $v \in C_c^\infty(\Omega)$, we estimate the diffusion term on $\Omega_0$:
\begin{equation}
    \begin{aligned}
        \left|\int_{\Omega_0} K(u^\epsilon_{n-1})\nabla u^\epsilon_n \cdot \nabla v \;d\textbf{x} \right| &= \left|\int_{\Omega_0} \sqrt{K(u^\epsilon_{n-1})}\nabla u^\epsilon_n \cdot \sqrt{K(u^\epsilon_{n-1})}\nabla v \;d\textbf{x}\right| \\
        & \leq \|\sqrt{K(u^\epsilon_{n-1})}\nabla u^\epsilon_n\|_{L^2(\Omega)} \|\nabla v\|_\infty  \|\sqrt{K(u^\epsilon_{n-1})}\|_{L^2(\Omega_0)} \to 0.
    \end{aligned}
\end{equation}

On the active sets $\Omega^m$ (fixed $m \geq 1$), we apply Egoroff's Theorem \cite{Royden1988}. For a fixed $\delta >0$, there exists a measurable subset $A_\delta \subset \Omega^m$ such that $|\Omega^m \setminus A_{\delta}| < \delta$ and $K(u^\epsilon_{n-1}) \to K(u_{n-1})$ uniformly on $A_\delta$.
Thus, there exists $\epsilon_0>0$ such that for all $\epsilon < \epsilon_0$, we have $K(u^\epsilon_{n-1}) \geq \frac{1}{2m}$ on $A_\delta$.
Consequently, the a priori estimate implies:
\begin{equation}
    \tau\| \nabla u^\epsilon_n \|^2 _{L^2(A_\delta)} \leq 2mC.
\end{equation}
This uniform bound on the gradient allows us to identify the weak limit on $A_\delta$. Since $u^\epsilon_n \rightharpoonup u'_n$ in $L^2(\Omega)$, we have:
\begin{equation}
    \nabla u^\epsilon_n \rightharpoonup \nabla u'_n \text{ weakly in } L^2(A_\delta).
\end{equation}
Combining the uniform convergence of the coefficient and the weak convergence of the gradient, we obtain for any $\delta >0$:
\begin{equation}
    \int_{A_{\delta}} K(u_{n-1}^\epsilon) \nabla u^\epsilon_n \cdot \nabla v\;d\textbf{x} \to \int_{A_{\delta}} K(u_{n-1}) \nabla u'_n \cdot \nabla v\;d\textbf{x}.
\end{equation}
Since $\delta$ is arbitrary, this convergence holds on $\Omega^m$:
\begin{equation}
    \int_{\Omega^m} K(u_{n-1}^\epsilon) \nabla u^\epsilon_n \cdot \nabla v\;d\textbf{x} \to \int_{\Omega^m} K(u_{n-1}) \nabla u'_n \cdot \nabla v\;d\textbf{x}.
\end{equation}
By exhausting $\Omega^+$ with the sequence $\Omega^m$, we conclude:
\begin{equation}
    \int_{\Omega^+} K(u_{n-1}^\epsilon) \nabla u^\epsilon_n \cdot \nabla v\;d\textbf{x} \to \int_{\Omega^+} K(u_{n-1}) \nabla u'_n \cdot \nabla v\;d\textbf{x},
\end{equation}
for all $v \in C_c^\infty(\Omega)$.
Applying a similar argument to $\sqrt{K}$ instead of $K$, we can identify $\bar Z = \sqrt \tau \sqrt{K(u_{n-1})}\nabla u'_n$ on $L^2(\Omega^+)$.

We now pass to the limit in the regularized equation \eqref{regdiffconv_prop} with a test function $v \in C^\infty_c(\Omega)$. Using the continuity and boundedness of the source term $\mathcal S$ and convection coefficient $\bar K$, the Lipschitz continuity of $\theta$, and the induction hypothesis, we obtain:
\begin{equation}
    \begin{aligned}
         (w_n- \theta(u_{n-1}),v) &+ \tau ((K(u_{n-1}))\nabla u'_n,\nabla v) \\
         &+ \tau(\bar K(u_{n-1}), \nabla v) = \tau (\mathcal S(u_{n-1}),v),\; \forall v \in C^\infty_c(\Omega).
    \end{aligned}
\end{equation}
Using the density of $C^\infty_c(\Omega)$ in the weighted space $V_n$, this variational equality extends to all test functions in $V_n$:
\begin{equation}
    \begin{aligned}
         (w_n- \theta(u_{n-1}),v) &+ \tau ((K(u_{n-1}))\nabla u'_n,\nabla v) \\
         &+ \tau(\bar K(u_{n-1}), \nabla v) = \tau (\mathcal S(u_{n-1}),v),\; \forall v \in V_n.
    \end{aligned}
\end{equation}
Finally, using Lemma \ref{Lemmabres}, we identify $w_n = \theta(u_n')$. By the uniqueness of the solution to the limit problem, we conclude $u_n = u'_n$.
Proceeding as in the base case, a further application of Lemma \ref{Lemmabres} proves the strong convergence:
\begin{equation}
    \|u_n^\epsilon - u_n \|_{L^2(\Omega)} \to 0.
\end{equation}
In all of the proof, all the constants above are independ of $\tau$, thus we can conclude that 
\begin{equation}
    \|u^{\tau,\epsilon} - u^\tau \| \to 0,
\end{equation}
as $\epsilon \to 0$, uniformly with respect to $\tau$. This holds because to prove the strong convergence with Lemma \ref{Lemmabres}, we only need to verify the inequality $\limsup_{\epsilon \to 0} (\theta(u^\epsilon_n),u^\epsilon_n) \leq (\theta(u_n),u_n)$. Since the regularization term $\tau \epsilon \|\nabla u^\epsilon_n\|^2$ is non-negative, it can be discarded from the energy equality, yielding an upper bound that depends solely on terms controlled by $\sqrt{\tau}\sqrt{K(u_{n-1}^\epsilon)}\nabla u_n^\epsilon$. Since $\sqrt{\tau}\|\sqrt{K(u_{n-1}^\epsilon)}\nabla u_n^\epsilon\|$ is bounded by a constant independent of $\tau$ (by the a priori estimates), the convergence argument holds uniformly for any $0<\tau\leq T$.
\end{proof}

\subsection{Convergence of the semi-implicit method in $L^2(\Omega)$}
We now demonstrate that the solution to the semi-implicit method converges to the continuous solution as the time step tends to zero.

\begin{proposition}
    The solution to equation \eqref{semidiscfull}, denoted by $u^\tau$, converges to the solution of \eqref{mainequation}, denoted by $u$, in $L^2(\Omega \times I)$ as $\tau \to 0$.
\end{proposition}

\begin{proof}
    The proof relies on the results established in Theorems \ref{theorem1} and \ref{regsolconvintime}, and Proposition \ref{convoftheregsoltodegsol_disc}.

    Let $\delta > 0$ be an arbitrary tolerance.
    First, we control the regularization errors. By Proposition \ref{convoftheregsoltodegsol_disc} (which guarantees convergence uniform in $\tau$) and Theorem \ref{theorem1}, there exists an $\epsilon_0 > 0$ such that:
    \begin{equation}
        \|u^{\tau,\epsilon_0} - u^\tau\| \leq \frac{\delta}{3} \quad \text{for all } \tau > 0,
    \end{equation}
    and
    \begin{equation}
        \|u^{\epsilon_0} - u\| \leq \frac{\delta}{3}.
    \end{equation}
    Next, we control the time discretization error for this fixed $\epsilon_0$. By Theorem \ref{regsolconvintime}, there exists a $\tau_0 > 0$ such that for all $0 < \tau < \tau_0$:
    \begin{equation}
        \| u^{\tau, \epsilon_0} - u^{\epsilon_0}\| \leq \frac{\delta}{3}.
    \end{equation}
    Finally, applying the triangle inequality, we obtain for all $0 < \tau < \tau_0$:
    \begin{equation}
        \|u - u^\tau\| \leq \|u - u^{\epsilon_0}\| + \|u^{\epsilon_0} - u^{\tau,\epsilon_0}\| + \|u^{\tau,\epsilon_0}-u^\tau\| \leq \frac{\delta}{3} + \frac{\delta}{3} + \frac{\delta}{3} = \delta.
    \end{equation}
    Since $\delta$ was arbitrary, this proves the convergence of the semi-implicit solution $u^\tau$ to the continuous solution $u$ in $L^2(\Omega \times I)$.
\end{proof}

\section{Numerical Validation of Theoretical Bounds}\label{sec10}
All numerical tests in this section are performed on a one-dimensional vertical infiltration column. The benchmark is solved in FreeFEM++ (v4.15) \cite{Hecht2012} with conforming $P_1$ elements, a mass-lumped time terms, L-scheme linearization, and UMFPACK for the linear subproblems. The complete setup is summarized in Table \ref{tab:sec10-setup}.

\begin{table}
    \centering
    \caption{Numerical setup for the validation test.}
    \label{tab:sec10-setup}
    \begin{tabular}{ll}
        \toprule
        Quantity & Value \\
        \midrule
        Domain & 1D vertical column $[0,100]$ \\
        Mesh & Uniform 1D mesh with $N=800$ segments (linear $P_1$ FEM) \\
        Boundary conditions & $\theta=1$ at $z=0$, $\theta=0$ at $z=100$ \\
        Initial condition & $\theta_0=1$ for $z<50$, $\theta_0=0$ for $z\ge 50$ \\
        Hydraulic parameters & $K_s=1$, $\alpha=0.02$, $\theta_s=0.48$, $\theta_r=0.08$, $n=2$ \\
        Derived parameters & $h_{\mathrm{cap}}=1/\alpha=50$, $m=1-1/n=0.5$ \\
        Time stepping (L-sweep) & $\Delta t=1$, $T=5$ \\
        Maximum-principle sweep & $L=0.52$, $\Delta t\in\{0.5,1,2,5\}$, $T=10$ \\
        Time-convergence study & $L=0.6$, $\Delta t\in\{0.5,0.25,0.125,0.0625\}$, $T=1$, reference $\Delta t_{\mathrm{ref}}=0.03125$ \\
        Nonlinear solver & L-scheme \\
        Linear solver & UMFPACK (direct solve at each L-iteration) \\
        Stopping criterion & Relative update error $<10^{-6}$, max 100 iterations/step \\
        Stabilization sweep & $L\in\{0.55,0.60,1.10\}$ \\
        \bottomrule
    \end{tabular}
\end{table}
Table \ref{tab:sec10-conv-L} reports convergence over the full simulation horizon for a minimal representative set of stabilization values. Here, ``Converged'' means that all time levels satisfy the nonlinear tolerance before reaching the iteration cap, while preserving $0\le\theta_h\le 1$ up to machine precision.

\begin{table}
    \centering
    \caption{L-scheme convergence for $\Delta t=1$ and $T=5$ on a minimal representative $L$ set.}
    \label{tab:sec10-conv-L}
    \begin{tabular}{cccc}
        \toprule
        $L$ & Converged & Max iterations (all steps) & Max final step error \\
        \midrule
        0.55 & Yes & 15 & $9.979420\times 10^{-7}$ \\
        0.60 & Yes & 16 & $9.574430\times 10^{-7}$ \\
        1.10 & Yes & 30 & $9.877250\times 10^{-7}$ \\
        \bottomrule
    \end{tabular}
\end{table}
For the discrete maximum-principle validation, Table \ref{tab:sec10-bounds-dt} summarizes the minimum and maximum saturation values over the full run for several time steps at fixed $L=0.52$. The numerical solution stays inside the admissible physical interval for all tested time steps.

\begin{table}
    \centering
    \caption{Maximum-principle check in the 1D benchmark ($L=0.52$, $T=10$).}
    \label{tab:sec10-bounds-dt}
    \begin{tabular}{cccc}
        \toprule
        $\Delta t$ & Final time reached & $\min_{z,t}\theta_h$ & $\max_{z,t}\theta_h$ \\
        \midrule
        0.5 & 10 & 0 & 1 \\
        1   & 10 & 0 & 1 \\
        2   & 10 & 0 & 1 \\
        5   & 10 & 0 & 1 \\
        \bottomrule
    \end{tabular}
\end{table}
To verify the main contribution on time discretization convergence, we performed a temporal refinement study at fixed $L=0.6$ and fixed spatial mesh ($N=800$). We computed the final-time $L^2$ error at $T=1$ against a reference solution obtained with $\Delta t_{\mathrm{ref}}=0.03125$.

\begin{table}
    \centering
    \caption{Empirical time-convergence study for the semi-implicit L-scheme (1D column benchmark).}
    \label{tab:sec10-timeconv}
    \begin{tabular}{cccc}
        \toprule
        $\Delta t$ & Absolute $L^2$ error & Relative $L^2$ error & Observed order \\
        \midrule
        0.5    & $1.018457\times 10^{0}$ & $1.518196\times 10^{-1}$ & -- \\
        0.25   & $9.549795\times 10^{-1}$ & $1.423571\times 10^{-1}$ & 0.0928 \\
        0.125  & $8.189412\times 10^{-1}$ & $1.220781\times 10^{-1}$ & 0.2217 \\
        0.0625 & $5.233871\times 10^{-1}$ & $7.802040\times 10^{-2}$ & 0.6459 \\
        \bottomrule
    \end{tabular}
\end{table}
The measured rates indicate clear temporal convergence in this degenerate setting while remaining sub-first-order on coarse-to-intermediate refinement pairs. This trend is consistent with reduced regularity near sharp fronts, where asymptotic first-order behavior may emerge only for finer time steps.

\begin{figure}[H]
    \centering
    \includegraphics[width=0.8\linewidth]{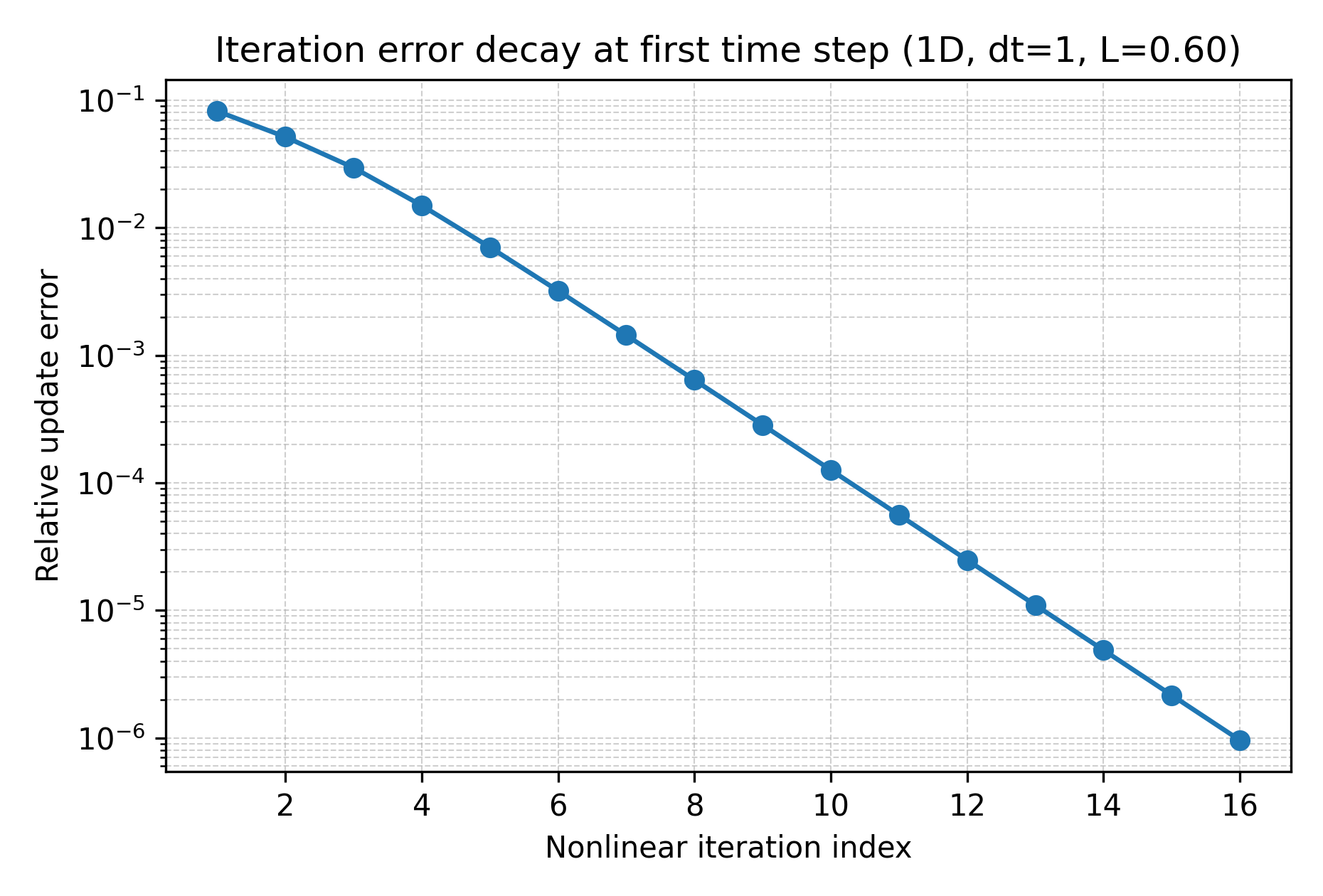}
    \caption{Empirical validation of Theorem 8.1 in 1D: L-scheme relative update error versus iteration at the first time step ($\Delta t=1$, $L=0.60$).}
    \label{fig:sec10-lscheme}
\end{figure}
\begin{figure}[H]
    \centering
    \includegraphics[width=0.78\linewidth]{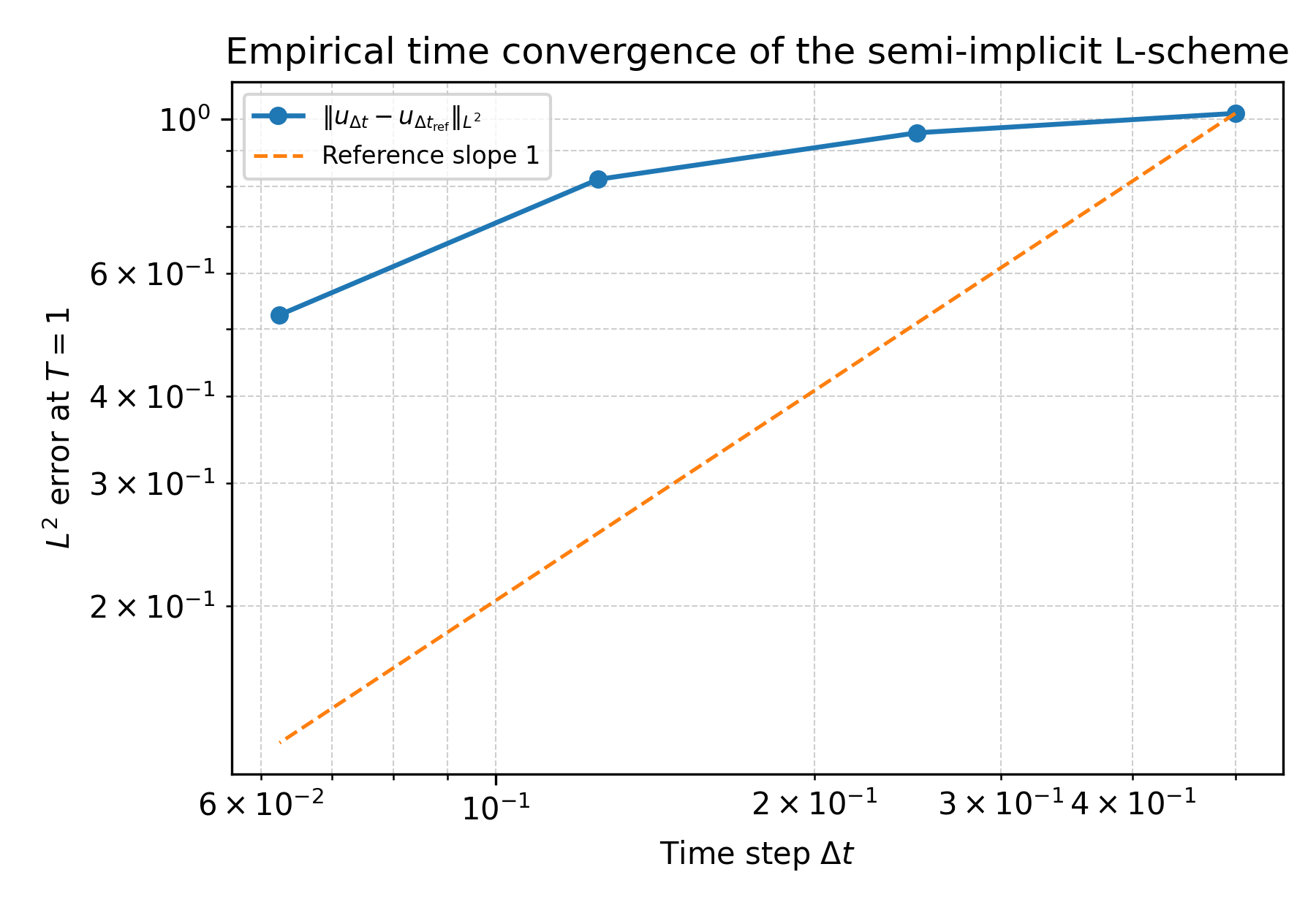}
    \caption{Log-log plot of final-time $L^2$ error versus $\Delta t$ for the 1D temporal refinement study, using $\Delta t_{\mathrm{ref}}=0.03125$ as reference.}
    \label{fig:sec10-timeconv}
\end{figure}
Table \ref{tab:sec10-conv-L} shows that the three representative choices $L=0.55$, $L=0.60$, and $L=1.10$ all converge for this benchmark, with final nonlinear errors below the prescribed tolerance. The near-threshold and moderate choices require fewer nonlinear iterations, while the larger stabilization remains robust but needs more iterations.
Figure \ref{fig:sec10-lscheme} confirms linear L-scheme convergence, with an average first-step error contraction factor of approximately $0.47$ per iteration. Figure \ref{fig:sec10-timeconv} and Table \ref{tab:sec10-timeconv} provide direct empirical support for temporal convergence of the semi-implicit discretization. Figure \ref{fig:sec10-maximum-principle} and Table \ref{tab:sec10-bounds-dt} show that the computed solution remains within the admissible physical range throughout infiltration, supporting the discrete maximum-principle behavior discussed in Section 6.

\begin{figure}[H]
    \centering
    \includegraphics[width=0.82\linewidth]{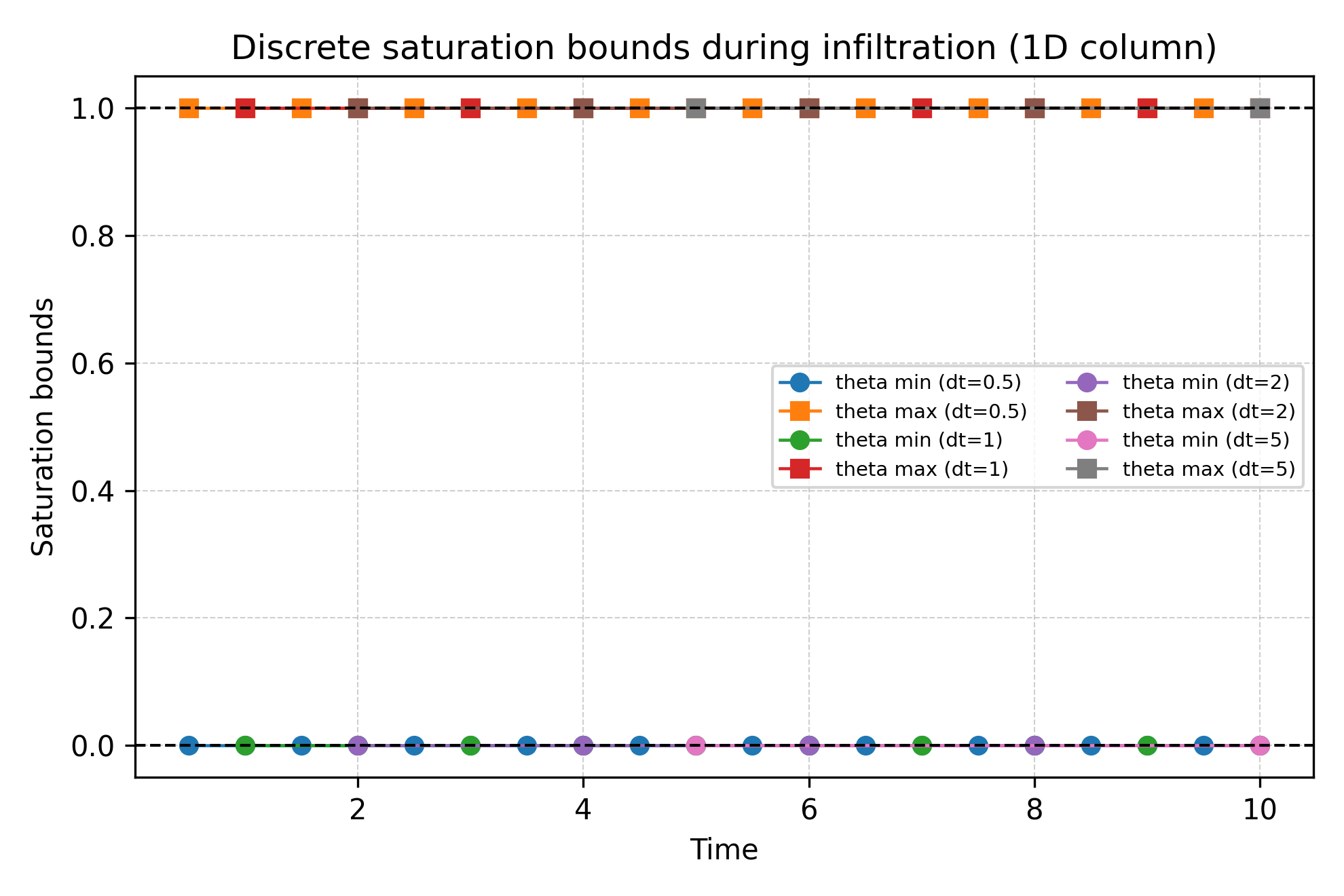}
    \caption{Empirical validation of the physical maximum principle in 1D: evolution of $\min\theta_h$ and $\max\theta_h$ for several time steps at fixed $L=0.52$.}
    \label{fig:sec10-maximum-principle}
\end{figure}

\section*{Acknowledgments}

This work was supported by an NSERC, Canada Discovery Grant (RGPIN-2019-06855) to Yves Bourgault 
and an NSERC, Canada Discovery Grant (RGPIN/5220-2022 \& DGECR/526-2022) to Abdelaziz Beljadid.

% We will prove a order of convergence for the regularized $K_\epsilon$, and we will drop the index $\epsilon$, and assuming the $\mathcal S=0$ for now. 
% Since $u$ is the weak solution of \eqref{uform}, it satisfies
% \begin{equation}
%     \bigg\langle \frac{\partial \theta(u(t))}{\partial t}, v\bigg\rangle + (\nabla \Phi(u(t)), \nabla v) + \bigg(\bar K(u(t)), \nabla v \bigg) = 0, \quad \forall v \in \mathcal{H}, \text{ a.e. } t \in I.
% \end{equation}
% Integrating over the time interval $(t_{n-1}, t_n)$, we obtain:
% \begin{equation}
%     (\theta(u(t_n)) - \theta(u(t_{n-1})), v) + \int_{t_{n-1}}^{t_n} (K(u(s))\nabla u(s), \nabla v) \, ds + \int_{t_{n-1}}^{t_n} \bigg(\bar K(u(s)), \nabla v \bigg) \, ds = 0, \quad \forall v \in \mathcal{H}.
% \end{equation}
% Furthermore, the semi-discrete equation \eqref{semidiscfull} admits a unique solution in $\mathcal H$ satisfying
% \begin{equation}
%     (\theta(u_n), v) - (\theta(u_{n-1}), v) + \tau (K(u_{n-1})\nabla u_n, \nabla v) + \tau\bigg(\bar K(u_{n-1}), \nabla v \bigg) = 0, \quad \forall v \in \mathcal H.
% \end{equation}
% Taking the difference between the equations we get:
% \begin{equation}
%     \begin{aligned}
%         &(\theta(u_n) -\theta(u(t_n)) -(\theta(u_{n-1}) - \theta(u(t_{n-1})),v) + \int_{t_{n-1}}^{t_n} (K(u_{n-1})\nabla u_n-K(u(s))\nabla u(s), \nabla v) \, ds\\
%         & \int_{t_{n-1}}^{t_n} \bigg(\bar K(u_{n-1})-\bar K(u(s)), \nabla v \bigg) \, ds
%     \end{aligned}
% \end{equation}

% \cite{Moussa2016}

% Preserve the submitted bibliography and citation numbering.


\begin{thebibliography}{10}
\expandafter\ifx\csname url\endcsname\relax
  \def\url#1{\texttt{#1}}\fi
\expandafter\ifx\csname urlprefix\endcsname\relax\def\urlprefix{URL }\fi
\expandafter\ifx\csname href\endcsname\relax
  \def\href#1#2{#2} \def\path#1{#1}\fi

\bibitem{Richards1931}
L.~A. Richards, Capillary conduction of liquids through porous mediums, Physics
  1~(5) (1931) 318--333.

\bibitem{AltLuckhaus1983}
H.~W. Alt, S.~Luckhaus, Quasilinear elliptic-parabolic differential equations,
  Math. Z. 183~(3) (1983) 311--341.

\bibitem{DiBenedetto1993}
E.~DiBenedetto, Degenerate Parabolic Equations, Universitext, Springer New
  York, New York, NY, 1993.

\bibitem{Otto1996}
F.~Otto, L1-contraction and uniqueness for quasilinear elliptic–parabolic
  equations, Journal of Differential Equations 131~(1) (1996) 20--38.

\bibitem{Carrillo1999}
J.~Carrillo, Entropy solutions for nonlinear degenerate problems, Archive for
  Rational Mechanics and Analysis 147~(4) (1999) 269--361.

\bibitem{ListRadu2016}
F.~List, F.~A. Radu, A study on iterative methods for solving {Richards}'
  equation, Computational Geosciences 20 (2016) 341--353.

\bibitem{Jones2001}
J.~E. Jones, C.~S. Woodward, Newton-krylov-multigrid solvers for large-scale,
  highly heterogeneous, variably saturated flow problems, Advances in Water
  Resources 24~(7) (2001) 763--774.

\bibitem{Diersch1999}
H.~J.~G. Diersch, P.~Perrochet, On the primary variable switching technique for
  simulating unsaturated-saturated flows, Adv. Water Resour. 23~(3) (1999)
  271--301.

\bibitem{Forsyth1995}
P.~Forsyth, Y.~Wu, K.~Pruess, Robust numerical methods for
  saturated-unsaturated flow with dry initial conditions in heterogeneous
  media, Adv. Water Resour. 18~(1) (1995) 25--38.

\bibitem{Wu2001}
Y.~S. Wu, P.~A. Forsyth, On the selection of primary variables in numerical
  formulation for modeling multiphase flow in porous media, J. Contam. Hydrol.
  48~(3--4) (2001) 277--304.

\bibitem{Brunner2012}
P.~Brunner, C.~T. Simmons, {HydroGeoSphere}: a fully integrated, physically
  based hydrological model, Ground Water 50~(2) (2012) 170--176.

\bibitem{Krabbenhoft2007}
K.~Krabbenhøft, An alternative to primary variable switching in
  saturated-unsaturated flow computations, Adv. Water Resour. 30~(3) (2007)
  483--492.

\bibitem{Zha2017}
Y.~Zha, J.~Yang, L.~Yin, Y.~Zhang, W.~Zeng, L.~Shi, A modified {Picard}
  iteration scheme for overcoming numerical difficulties of simulating
  infiltration into dry soil, J. Hydrol. 551 (2017) 56--69.

\bibitem{Maina2017}
F.~H. Maina, P.~Ackerer, Ross scheme, {Newton--Raphson} iterative methods and
  time-stepping strategies for solving the mixed form of {Richards'} equation,
  Hydrol. Earth Syst. Sci. 21~(6) (2017) 2667--2683.

\bibitem{Kees2002}
C.~E. Kees, C.~T. Miller, Higher order time integration methods for two-phase
  flow, Adv. Water Resour. 25~(2) (2002) 159--177.

\bibitem{Zha2019}
Y.~Zha, J.~Yang, J.~Zeng, C.-H.~M. Tso, W.~Zeng, L.~Shi, Review of numerical
  solution of {Richardson--Richards} equation for variably saturated flow in
  soils, WIREs Water 6 (2019) e1364.

\bibitem{Schweizer2007}
B.~Schweizer, Regularization of outflow problems in unsaturated porous media
  with dry regions, J. Differential Equations 237~(2) (2007) 278--306.

\bibitem{Pop2011}
I.~S. Pop, B.~Schweizer, Regularization schemes for degenerate {Richards}
  equations and outflow conditions, Math. Models Methods Appl. Sci. 21~(8)
  (2011) 1685--1712.

\bibitem{Fevotte2024}
F.~F{\'e}votte, A.~Rappaport, M.~Vohral{\'\i}k, Adaptive regularization for the
  {R}ichards equation, Comput. Geosci. 28 (2024) 1371--1388.

\bibitem{Pop2004}
I.~Pop, F.~Radu, P.~Knabner, Mixed finite elements for the {Richards}’
  equation: linearization procedure, Journal of Computational and Applied
  Mathematics 168~(1) (2004) 365--373, selected Papers from the Second
  International Conference on Advanced Computational Methods in Engineering
  (ACOMEN 2002).

\bibitem{Slodicka2002}
M.~Slodi{\v{c}}ka, A robust and efficient linearization scheme for doubly
  nonlinear and degenerate parabolic problems arising in flow in porous media,
  SIAM Journal on Scientific Computing 23~(5) (2002) 1593--1614.

\bibitem{Mitra2019}
K.~Mitra, I.~S. Pop, A modified {L-scheme} to solve nonlinear diffusion
  problems, Comput. Math. Appl. 77~(6) (2019) 1722--1738.

\bibitem{Kamil2024}
H.~Kamil, A.~Beljadid, A.~Soulaïmani, Y.~Bourgault, Semi-implicit schemes for
  modeling water flow and solute transport in unsaturated soils, Advances in
  Water Resources 193 (2024) 104835.

\bibitem{Paniconi1991}
C.~Paniconi, M.~Putti, G.~F. Pinder, Numerical evaluation of iterative and
  noniterative methods for the solution of the nonlinear {Richards} equation,
  Water Resources Research 27~(6) (1991) 1147--1163.

\bibitem{Keita2021}
S.~Keita, A.~Beljadid, Y.~Bourgault, Implicit and semi-implicit second-order
  time stepping methods for the richards equation, Advances in Water Resources
  148 (2021) 103841.

\bibitem{Celia1990}
M.~A. Celia, E.~T. Bouloutas, R.~L. Zarba, A general mass-conservative
  numerical solution for the unsaturated flow equation, Water Resour. Res.
  26~(7) (1990) 1483--1496.

\bibitem{benfanich2025}
A.~Benfanich, Y.~Bourgault, A.~Beljadid, A finite element method using a
  bounded auxiliary variable for solving the richards equation (2025).
\newblock \href {http://arxiv.org/abs/2510.13012} {\path{arXiv:2510.13012}}.

\bibitem{Rektorys1982}
K.~Rektorys, The method of discretization in time and partial differential
  equations, in: Equadiff 5, BSB B.G. Teubner Verlagsgesellschaft, 1982, pp.
  293--296.

\bibitem{Kacur1990}
J.~Ka{\v{c}}ur, Method of rothe in evolution equations, in: J.~Vosmansk{\'y},
  M.~Zl{\'a}mal (Eds.), Equadiff 6, Springer Berlin Heidelberg, Berlin,
  Heidelberg, 1986, pp. 23--34.

\bibitem{Cavalheiro2008WeightedSobolev}
A.~C. Cavalheiro, Weighted sobolev spaces and degenerate elliptic equations,
  Boletim da Sociedade Paranaense de Matemática (3s.) 26~(1-2) (2008)
  117--132.

\bibitem{brezis1973ope}
H.~Brezis, Ope¦rateurs maximaux monotones et semi-groupes de contractions dans
  les espaces de Hilbert, North-Holland Mathematics Studies, North Holland,
  1973.

\bibitem{LaxMilgram+1955+167+190}
P.~D. Lax, A.~N. Milgram, IX. Parabolic Equations, Princeton University Press,
  Princeton, 1955, pp. 167--190.

\bibitem{Agarwal2018}
P.~Agarwal, M.~Jleli, B.~Samet, Fixed Point Theory in Metric Spaces: Recent
  Advances and Applications, 1st Edition, Springer Singapore, Singapore, 2018.

\bibitem{Brezis2011}
H.~Brezis, Functional Analysis, Sobolev Spaces and Partial Differential
  Equations, Springer New York, New York, NY, 2011.

\bibitem{LionsMagenes1972}
J.~L. Lions, E.~Magenes, Non-Homogeneous Boundary Value Problems and
  Applications, Vol.~I of Die Grundlehren der mathematischen Wissenschaften,
  Springer-Verlag Berlin Heidelberg New York, 1972.

\bibitem{Aubin1963}
J.-P. Aubin, Un th{\'e}or{\`e}me de compacit{\'e}, Comptes Rendus de
  l'Acad{\'e}mie des Sciences de Paris 256 (1963) 5042--5044.

\bibitem{Royden1988}
H.~L. Royden, Real Analysis, Prentice Hall, 1988.

\bibitem{Ern2021}
A.~Ern, J.-L. Guermond, Finite Elements III: First-Order and Time-Dependent
  PDEs, Springer International Publishing, Cham, 2021.

\bibitem{Moussa2016}
A.~Moussa, {Some variants of the classical Aubin--Lions Lemma}, Journal of
  Evolution Equations 16~(1) (2016) 65--93.

\bibitem{VanGenuchten1980}
M.~T. van Genuchten, A closed-form equation for predicting the hydraulic
  conductivity of unsaturated soils, Soil Sci. Soc. Am. J. 44~(5) (1980)
  892--898.

\bibitem{mualem1976}
Y.~Mualem, A new model for predicting the hydraulic conductivity of unsaturated
  porous media, Water Resour. Res. 12~(3) (1976) 513--522.

\bibitem{Hecht2012}
F.~Hecht, New developments in {FreeFem++}, J. Numer. Math. 20~(3--4) (2012)
  251--266.

\end{thebibliography}
\end{document}